\documentclass[12pt]{amsart}
\usepackage{setspace}
\usepackage{hyperref}
\usepackage[alphabetic]{amsrefs}
\usepackage{mathtools}
\usepackage{enumerate}
\usepackage{amsmath, amsthm, amssymb, thmtools, etoolbox}
\usepackage{pgfplots}
\usepackage{amsaddr}
\usepackage{comment}

\usepackage{mathrsfs}
\usetikzlibrary{arrows}
\usepackage{tikz-cd}
\usetikzlibrary{calc}
\allowdisplaybreaks

\theoremstyle{plain}
\newtheorem{thm}{Theorem}[section]
\newtheorem{defn}[thm]{Definition}
\newtheorem{rem}[thm]{Remark}
\newtheorem{prop}[thm]{Proposition}
\newtheorem{cor}[thm]{Corollary}
\newtheorem{lem}[thm]{Lemma}
\newtheorem{nota}[thm]{Notation}

\numberwithin{equation}{section}

\newcommand{\bigslant}[2]{{\raisebox{.2em}{$#1$}\left/\raisebox{-.2em}{$#2$}\right.}}
\newcommand{\midarrow}{\tikz \draw[-triangle 90] (0,0) -- +(.1,0);}
\newcommand{\backmidarrow}{\tikz \draw[-triangle 90] (0,0) -- +(-.1,0);}
\newcommand{\slantmidarrow}{\tikz \draw[-triangle 90] (0,0) -- +(.1,.1);}
\newcommand{\backslantmidarrow}{\tikz \draw[-triangle 90] (0,0) -- +(-.1,-.1);}
\newcommand{\opslantmidarrow}{\tikz \draw[-triangle 90] (0,0) -- +(.1,-.1);}
\newcommand{\backopslantmidarrow}{\tikz \draw[-triangle 90] (0,0) -- +(-.1,.1);}
\newcommand{\upmidarrow}{\tikz \draw[-triangle 90] (0,0) -- +(0,.1);}
\newcommand{\downmidarrow}{\tikz \draw[-triangle 90] (0,0) -- +(0,-.1);}

\begin{document}
	\begin{abstract}
		We introduce various notions of quantum symmetry in a directed or undirected multigraph with no isolated vertex and explore relations among them. If the multigraph is single edged (that is, a simple graph where loops are allowed), all our notions of quantum symmetry reduce to already existing notions of quantum symmetry provided by Bichon and Banica. Our constructions also show that any multigraph with at least two pairs of vertices with multiple edges among them possesses genuine quantum symmetry.
	\end{abstract}
	\author{Debashish Goswami}
	\thanks{It is to be noted that Debashish Goswami is partially supported by J.C. Bose national fellowship awarded by D.S.T., Government of India.}
	\address{Statistics and Mathematics Unit, Indian Statistical Institute\\ 203, B.T. Road, Kolkata 700108, India\\\textnormal{email: \texttt{debashish\_goswami@yahoo.co.in}}}

	\author{Sk Asfaq Hossain}
	\address{Statistics and Mathematics Unit, Indian Statistical Institute\\ 203, B.T. Road, Kolkata 700108, India\\\textnormal{email: \texttt{asfaq1994@gmail.com}}}

	\title{Quantum Symmetry in multigraphs (part I)}
	\maketitle
\section{Introduction}\label{Intro}
The idea of quantum groups was introduced by Drinfeld and Jimbo (\cite{Drinfelprimed1987}, \cite{Drinfelprimed1989}, \cite{Jimbo1985}). It was done on an algebraic level where quantum groups were viewed as Hopf algebras typically arising as deformations of semisimple Lie algebras. The analytic version of quantum groups was first described by Woronowicz (\cite{Woronowicz1987}, \cite{Woronowicz1998}) who formulated the notion of \textbf{compact quantum group} as a generalization of a compact topological group in the noncommutative realm.  \par 
Groups are often viewed as ``symmetry objects”, in a similar way, quantum groups
correspond to some kind of ``generalized symmetry” of physical systems and mathematical structures. Indeed, the idea of a group acting on a space can be extended to the idea of a quantum group co-acting on a non commutative space (that is, possibly a non commutative C*algebra). The question of defining and finding ``all quantum symmetries” arises naturally in this context. Study of quantum symmetry in analytic setting, that is, in the framework of compact quantum groups was started by Shuzhou Wang. In his seminal work \cite{Wang1998}, Wang introduced  notion of \textbf{quantum permutations} (in the category of compact quantum groups)  of  $n$ objects  and defined \textbf{quantum permutation group} $S^+_n$ as the universal object in the category of all such \textbf{quantum permutations}. The quantum group $S^+_n$ is indeed the compact quantum analogue of the standard permutation group $S_n$ on $n$ elements.
\par 
Following the introduction of quantum automorphisms of a finite set, a logical progression led to an investigation of the concepts of quantum automorphisms of finite graphs and small metric spaces. In (\cite{Bichon2003}), Bichon introduced a the notion of quantum automorphism in a finite directed single edged graph $(V,E)$ which was formulated in terms of simultaneous quantum permutations of both edge set $E$  and vertex set $V$. These permutations were compatible through source and target maps of the directed graph. Here by \textbf{single edged}, we mean a \textbf{simple} graph where loops are allowed. A loop is an edge with a single endpoint vertex. Two years later, in \cite{Banica2005} Banica gave a more general description of quantum symmetry in a single edged graph in terms of its adjacency matrix. Any quantum permutation of vertex set which commutes with the adjacency matrix is a quantum automorphism of the single edged graph in Banica's sense. As there was absolutely no restriction on the entries of the adjacency matrix, this construction was generalised easily to produce quantum automorphisms in the context of  weighted graphs and small metric spaces (\cite{Banica2005a}). For a single edged graph $(V,E)$, the categories of quantum automorphisms described by Bichon and Banica will be denoted as $\mathcal{D}^{Ban}_{(V,E)}$ and $\mathcal{D}^{Bic}_{(V,E)}$ respectively. It turned out that $\mathcal{D}^{Bic}_{(V,E)}$ is always a full subcategory of $\mathcal{D}^{Ban}_{(V,E)}$.
\par

It is natural to ask whether Banica and Bichon's notions of quantum automorphisms can be generalised in the context of multigraphs. A multigraph or a finite quiver  $(V,E)$ consists of a finite vertex set $V$ and a finite edge set $E$ with source and target maps $s:E\rightarrow V$ and $t:E\rightarrow V$. Classically an automorphism of a multigraph is pair $(f_V, f_E)$ where $f_V$ and $f_E$ are permutations of vertex set and edge set respectively which are compatible via source and target maps $s$ and $t$. In single edged case, the formulations of quantum symmetry were done in terms of permutation of vertices and adjacency relations between two vertices. Different technique needs to be adapted for multigraphs as an edge is not uniquely determined by the adjacency relations.
\par 
We have reformulated  the notions of quantum symmetry in terms of ``permutations" of edges instead of permutations of vertices which is useful in the context of multigraphs.
\par 
For a multigraph $(V,E)$, we have constructed three different categories  $\mathcal{C}^{Ban}_{(V,E)}$, $\mathcal{C}^{sym}_{(V,E)}$ and $\mathcal{C}^{Bic}_{(V,E)}$ consisting of compact quantum groups co-acting by preserving different levels of quantum symmetry in $(V,E) $.
If $(V,E)$ is single edged, it turns out that $\mathcal{C}^{Ban}_{(V,E)}= \mathcal{C}^{sym}_{(V,E)}\cong \mathcal{D}^{Ban}_{(V,E)}$ and $\mathcal{C}^{Bic}_{(V,E)}\cong \mathcal{D}^{Bic}_{(V,E)}$. For a multigraph $(V,E)$, we have the following :
\begin{equation*}
	\mathcal{C}^{Bic}_{(V,E)}\subseteq \mathcal{C}^{sym}_{(V,E)} \subseteq \mathcal{C}^{Ban}_{(V,E)}.
\end{equation*}
It can be easily seen that the categories  $\mathcal{C}^{Bic}_{(V,E)}$ and $\mathcal{C}^{Ban}_{(V,E)}$  admit universal objects namely $Q^{Bic}_{(V,E)}$ and $Q^{Ban}_{(V,E)}$. However that is not the case for   $\mathcal{C}^{sym}_{(V,E)}$. It is still unclear whether for an arbitrary multigraph $(V,E)$, the category $\mathcal{C}^{sym}_{(V,E)}$ admits a universal object or not.
The compact quantum group $Q^{Bic}_{(V,E)}$ is the \textbf{quantum automorphism group} of $(V,E)$ which is a  quantum analogue of the  classical automorphism group of $(V,E)$. On the other hand, $Q^{Ban}_{(V,E)}$ is too large to be called an automorphism group of $(V,E)$ and therefore will be referred to as \textbf{universal quantum group associated with $(V,E)$}. This is precisely the reason for considering a smaller category $\mathcal{C}^{sym}_{(V,E)}$ to make a true generalisation of Banica's notion of quantum symmetry.
\par 
It is natural to ask for which class of multigraphs, the two categories $\mathcal{C}^{sym}_{(V,E)}$ and $\mathcal{C}^{Bic}_{(V,E)}$ coincide. We have provided a necessary and sufficient condition in terms of weighted symmetry of the \textbf{underlying weighted graph}. We have shown that the categories $\mathcal{C}^{sym}_{(V,E)}$ and $\mathcal{C}^{Bic}_{(V,E)}$ coincide if and only if Banica and Bichon's notion of quantum symmetry coincide for the \textbf{underlying weighted graph} of $(V,E)$. For this class of multigraphs, the compact quantum group $Q^{Bic}_{(V,E)}$ does act as a universal object in $\mathcal{C}^{sym}_{(V,E)}$. For uniform multigraphs, that is, a multigraph with either zero or fixed number of edges between two vertices,  We have expressed $Q^{Bic}_{(V,E)}$ as \textbf{free wreath product by quantum permutation groups} (\cite{Bichon2004}, \cite{Banica2007a}) where the co-action corresponding to the wreath product comes from a permutation of pairs of vertices induced by the weighted symmetry of the \textbf{underlying single edged graph}. This wreath product formula for $Q^{Bic}_{(V,E)}$ also emphasizes that  any multigraph which have at least two pairs of vertices with multiple edges among them possesses genuine quantum symmetry.
\par 
There has been an extensive study of quantum symmetry in \textbf{graph C* algebras} in recent times (see \cite{Banica2013},\cite{Joardar2018}, \cite{Schmidt2018},\cite{Brannan2022} and references therein). Following the line of \cite{Schmidt2018}  we have shown that our notions of quantum symmetry in multigraphs in fact lift to the level of graph C* algebras. Apart from mathematical structures, multigraphs are also important in many physical models
such as lattices of atoms with double or triple bonds.
\par 
Now we briefly discuss layout of this paper. In section  \ref{chap_prelim} we provide necessary prerequisites about graphs, compact quantum groups and co-actions of compact quantum groups on single edged graphs. We also introduce a set of important notations at the end of this section that we will be using throughout this article. In section \ref{chap_quan_sym_direc_mult}, we observe equivalent descriptions of right and left equivariant bi-unitary co-representations on the edge space which essentially give us the formulas to capture permutations of the vertex set in terms of ``permutation" of edges (see theorems \ref{bi-mod_equiv_source}, \ref{bi-mod_equiv_targ} and \ref{source_targ_consis_thm}). Using these revelations, we introduce various notions of quantum symmetry in a multigraph. Various compact quantum groups associated with a multigraph are also described here (see definition \ref{Q^Ban_def}). To overcome inadequacy of $Q^{Ban}_{(V,E)}$ to act as an automorphism group of the multigraph and  make true generalisation of Banica's notion of quantum symmetry we introduce the notion of ``restricted orthogonality" (see definition \ref{maindef_symmetric}) and explore its various consequences. One important consequence would be that any action satisfying ``restricted orthogonality" preserves uniform components of a multigraph (see proposition \ref{direct_sum}). The wreath product formula related to the quantum automorphism group of a multigraph In Bichon's sense is also described in this section. Section \ref{graphC*algebra} is dedicated to quantum symmetry of graph C* algebras associated with multigraphs. In Section \ref{quan_sym_undir_mult} we talk about ``undirected" multigraphs and briefly discuss how our work in directed setting can be used to describe quantum symmetry in ``undirected" multigraphs.  
	
	\section{\textbf{Preliminaries}}\label{chap_prelim}
	\subsection{Finite quivers or multigraphs:}\label{graphs}
	We recall the notions of finite quivers and morphisms among them. For more details on quivers and path algebras see \cite{Grigoryan2018}.	 
	\begin{defn}
		A \textbf{finite quiver} or a \textbf{multigraph} $(V,E)$ consists of a finite set of vertices $V$ and a finite set of edges $E$ with  source and  target maps $s:E\rightarrow V$ and $t:E\rightarrow V$.
		\par
		An edge $\tau\in E$ is called a \textbf{loop} if $s(\tau)=t(\tau)$. We will denote $L\subseteq E$ to be the set of all ``loops" in $(V,E).$ 
		\par 
		The adjacency matrix $W=(W^i_j)_{i,j\in V}$ is given by $W^i_j=|\{\tau\in E|s(\tau)=i,t(\tau)=j\}|$. Here $|.|$ denotes cardinality of a set.
		\par 
	\end{defn}
\begin{defn}
	A multigraph $(V,E)$ is called a \textbf{single edged graph} if $W^i_j=1$ or $0$ for all $i,j\in V$. In case of a single edged graph, the edge set $E$ can be identified with a subset of $V\times V$. \par 
	A \textbf{weighted single edged graph} $(V,E,w)$ is a single edged graph $(V,E)$ with a weight function $w:E\rightarrow \mathbb{C}$ on the set of edges. In this case, the adjacency matrix is defined to be $W^i_j=w((i,j))$ if $(i,j)\in E$ and $0$ otherwise.  
\end{defn}
	
	\begin{defn}\label{under_simp}
		For a multigraph $(V,E)$, the \textbf{underlying single edged graph} $(V,\overline{E})$ is the single edged graph with same vertex set $V$ and a set of edges $\overline{E}$ given by, $$\overline{E}:=\{(i,j)\in V\times V|W^i_j\neq 0\}.$$ 
	\end{defn}
	\begin{defn}\label{under_weighted_simp}
		For a multigraph $(V,E)$, the \textbf{underlying weighted single edged graph} is the underlying single edged graph $(V,\overline{E})$ with a weight function $w:\overline{E}\rightarrow \mathbb{C}$ defined by $w((i,j))=|E^i_j|$. 
	\end{defn}
	\subsection*{\textbf{Morphisms of finite quivers or multigraphs:}}
	We recall definition 2.3 from \cite{Grigoryan2018}. For more detailed discussion on different automorphisms of a multigraph, see also \cite{Gela1982}.
	\begin{defn}\label{class_aut_mult_def}
		Let $(V,E)$ and $(V',E')$ be two finite quivers or multigraphs with pairs of source and target maps given by $(s,t)$ and $(s',t')$ respectively. A \textbf{morphism of quivers} $f:(V,E)\rightarrow (V',E')$ is a pair of maps $(f_V,f_E)$ where $f_V:V\rightarrow V'$ is a map of vertices and $f_E:E\rightarrow E'$ is a map of edges satisfying,
		\begin{equation*}
			f_V(s(\tau))=s'(f_E(\tau))\quad\text{and}\quad f_V(t(\tau))=t'(f_E(\tau))\quad\text{for all}\quad\tau\in E.
		\end{equation*}
	An \textbf{automorphism of a finite quiver  or a multigraph} $(V,E)$ is an invertible morphism from  $(V,E)$ to $(V,E)$. The collection of all such automorphisms is the \textbf{classical automorphism group} of $(V,E)$ and is denoted as $G^{aut}_{(V,E)}$. 
	\end{defn}
	
	\subsection{Compact quantum group:}\label{CQG}
	We give a brief description of compact quantum groups and related concepts. For detailed discussion on quantum groups, see \cite{Chari1995}, \cite{Maes1998},  \cite{Timmermann2008}, \cite{Neshveyev2013} and \cite{Goswami2016}. All C* algebras here will be assumed to be unital and all tensor products will be minimal tensor product of C* algebras unless explicitly mentioned otherwise.
	\begin{defn}
		A compact quantum group or a CQG (in short) is a pair $(\mathcal{A},\Delta)$ where $\mathcal{A}$ is a unital C* algebra and $\Delta:\mathcal{A}\rightarrow\mathcal{A}\otimes\mathcal{A}$ is a homomorphism of C* algebras satisfying the following conditions:
		\begin{enumerate}
			\item $(\Delta\otimes id)\Delta=(id\otimes \Delta)\Delta$ (coassociativity).
			\item Each of the linear spans of $\Delta(\mathcal{A})(1\otimes\mathcal{A})$ and $\Delta(\mathcal{A})(\mathcal{A}\otimes 1)$ is norm-dense in $\mathcal{A}\otimes\mathcal{A}$.
		\end{enumerate} 
	\end{defn}
	It is known that there exists a unique Haar state on a compact quantum group which is the non-commutative analogue of Haar measure on a classical compact group.
	\begin{defn}
		The Haar state $h$  on a compact quantum group $(\mathcal{A},\Delta)$ is the unique state on $\mathcal{A}$ which satisfies the following conditions:
		\begin{equation*}
			(h\otimes id)\Delta(a)=h(a)1_{\mathcal{A}} \quad \text{and}\quad (id\otimes h)\Delta(a)=h(a)1_{\mathcal{A}}
		\end{equation*}
		for all $a\in \mathcal{A}$.
	\end{defn}
	\begin{defn}
		A quantum group homomorphism $\Phi$ among two compact quantum groups $(\mathcal{A}_1,\Delta_1)$ and $(\mathcal{A}_2,\Delta_2)$ is a C* algebra homomorphism $\Phi:\mathcal{A}_1\rightarrow\mathcal{A}_2$ satisfying the following condition:
		\begin{equation*}
			(\Phi\otimes\Phi)\circ\Delta_1=\Delta_2\circ\Phi.
		\end{equation*}

	\end{defn}
	\begin{defn}
		A Woronowicz C* subalgebra of a compact quantum group $(\mathcal{A},\Delta)$ is a C* subalgebra $\mathcal{A}'$ such that $(\mathcal{A}',\Delta|_{\mathcal{A}'})$ is a compact quantum group and the inclusion map $i:\mathcal{A}'\rightarrow \mathcal{A}$ is a homomorphism of compact quantum groups. 
	\end{defn}
	\begin{defn}
		A Woronowicz C* ideal of a compact quantum group $(\mathcal{A},\Delta)$ is a two sided C*ideal $\mathcal{I}$ such that $\Delta(\mathcal{I})\subseteq ker(\pi\otimes \pi)$ where $\pi$ is the natural quotient map $\pi:\mathcal{A}\rightarrow \mathcal{A}/\mathcal{I}$.
		
	\end{defn}
	\begin{prop}
		The quotient of a compact quantum group $(\mathcal{A},\Delta)$ by a Woronowicz C* ideal $\mathcal{I}$ has a unique compact quantum group structure such that the quotient map $\pi$ is a homomorphism of compact quantum groups. More precisely, the co-product $\tilde{\Delta}$ on $\mathcal{A}/\mathcal{I}$ is given by,
		\begin{equation*}
			\tilde{\Delta}(a+\mathcal{I})=(\pi\otimes \pi)\Delta(a) 
		\end{equation*}
		where $a\in \mathcal{A}$.
	\end{prop}
	\begin{defn}
		A compact quantum group $(\mathcal{A}',\Delta')$ is said to be quantum subgroup of another compact quantum group $(\mathcal{A},\Delta)$ if there exists a Woronowicz C* ideal $\mathcal{I}$ such that $(\mathcal{A}',\Delta')\cong(\mathcal{A},\Delta)/\mathcal{I}$.
	\end{defn}
	\subsubsection{\textbf{Co-actions and co-representations:}}
	\begin{defn}
		Let $H$ be a finite dimensional Hilbert space and $(\mathcal{A},\Delta)$ be a compact quantum group. We consider the Hilbert $\mathcal{A}$-module $H\otimes\mathcal{A}$ with induced $\mathcal{A}$-valued inner product from $H$. A finite dimensional co-representation of $(\mathcal{A},\Delta)$ on $H$ is a $\mathbb{C}$-linear map $\delta:H\rightarrow H\otimes \mathcal{A}$ such that $\tilde{\delta}\in B(H)\otimes \mathcal{A}$ given by $\tilde{\delta}(\xi\otimes a)=\delta(\xi)a$ ($\xi\in H$,$a\in\mathcal{A}$) satisfies the following condition:
		\begin{equation*}
			(id\otimes \Delta)\tilde{\delta}=\tilde{\delta}_{(12)}\tilde{\delta}_{(13)}
		\end{equation*}
		where $\tilde{\delta}_{(12)}$ and $\tilde{\delta}_{(13)}$ are common leg notations defined in section 5 of \cite{Maes1998}.
		\begin{rem}
			By choosing  an orthonormal basis $\{e_1,..,e_n\}$ of $H$ we can identify $H$ with $\mathbb{C}^n$ and $B(H)$ with $M_n(\mathbb{C})$. For a $\mathbb{C}$-linear map $\delta:H\rightarrow H\otimes \mathcal{A}$, we define $U^\delta\in M_n(\mathbb{\mathcal{A}})$ by    $(U^{\delta})_{ij}=<e_i\otimes 1_{\mathcal{A}},\delta(e_j)>_{\mathcal{A}}$. It is clear that $\delta$ is uniquely determined by the matrix $U^{\delta}$ and is a co-representation if and only if $$\Delta(U^\delta_{ij})=\sum_{k=1}^nU^\delta_{ik}\otimes U^\delta_{kj}.$$ $U^{\delta}$ is said to be the \textbf{co-representation matrix} of $\delta$.  Later in this article we might also write the coefficients of a co-representation matrix as $(U^{\delta})^i_j$ instead of $(U^{\delta})_{ij}$ for notational ease and convenience. 
			\par 
			A co-representation $\delta$ is said to be \textbf{non-degenerate} if $U^{\delta}$ is invertible in $M_n(\mathcal{A})$ and \textbf{unitary} if the matrix $U^{\delta}$ is unitary in $M_n(\mathcal{A})$, that is, $U^{\delta}{U^{\delta}}^*={U^{\delta}}^*U^{\delta}=Id_{M_n(\mathcal{A})}$. 
		\end{rem} 	
	\end{defn}

	\begin{defn}
		For a finite dimensional co-representation $\delta$ of a compact quantum group $(\mathcal{A},\Delta)$ the contragradient co-representation $\overline{\delta}$ is  defined by the co-representation matrix $\overline{U^{\delta}}$, where $\overline{U^{\delta}_{ij}}={U^{\delta}_{ij}}^*$. \par 
    As we have identified co-representations with operator valued matrices and will be working only with finite dimensional co-representations, we will consider contragradient representation on the same finite dimensional Hilbert space instead of its dual. 
	\end{defn}
	It is known from representation theory of compact quantum groups that for a compact quantum group $(\mathcal{A},\Delta)$, there is a dense subalgebra $\mathcal{A}_0$ generated by the matrix elements of its finite dimensional co-representations. This subalgebra $\mathcal{A}_0$ with the co-product $\Delta|_{\mathcal{A}_0}$ is a Hopf * algebra in its own right and referred to as \textbf{underlying Hopf * algebra of matrix elements} of $(\mathcal{A},\Delta)$. The Haar state $h$ is faithful on $\mathcal{A}_0$ and is tracial if $(\mathcal{A},\Delta)$ is a compact quantum group of Kac type (see proposition 1.7.9 in \cite{Neshveyev2013} ).
	\par 
	Now we describe the notion of a co-action of a compact quantum group on a unital C* algebra.
	\begin{defn}
		Let $\mathcal{B}$ be a unital C* algebra. A co-action of a compact quantum group $(\mathcal{A},\Delta)$ on $\mathcal{B}$ is a C* homomorphism $\alpha:\mathcal{B}\rightarrow \mathcal{B}\otimes\mathcal{A}$ satisfying the following conditions:
		\begin{enumerate}
			\item $(\alpha\otimes id)\alpha=(id\otimes\Delta)\alpha$.
			\item Linear span of $\alpha(\mathcal{B})(1_{\mathcal{B}}\otimes\mathcal{A})$ is norm-dense in $\mathcal{B}\otimes\mathcal{A}$. 
		\end{enumerate}
		A co-action $\alpha$ is said to be \textbf{faithful} if there does not exist a proper Woronowicz C* algebra $\mathcal{A}'$ of $(\mathcal{A},\Delta)$ such that $\alpha$ is also a co-action of $(\mathcal{A}',\Delta|_{\mathcal{A}'})$ on $\mathcal{B}$.
	\end{defn}
	For a unital C* algebra $\mathcal{B}$, we consider the category of quantum transformation groups whose objects are compact quantum groups co-acting on $\mathcal{B}$ and morphisms are quantum group homomorphisms intertwining such co-actions. The universal object in this category, if it exists (it might not, for example see \cite{Wang1998} for example), is said to be \textbf{quantum automorphism group} of $\mathcal{B}$. The following proposition will be crucial to our constructions later on.
	
	\begin{prop}\label{adj_co-action}
		For a finite dimensional unitary co-representation $\delta:H\rightarrow H\otimes \mathcal{A}$ of a compact quantum group $(\mathcal{A},\Delta)$, there is a co-action $Ad_{\delta}:B(H)\rightarrow B(H)\otimes \mathcal{A}$ of $(\mathcal{A},\Delta)$ on the algebra  $B(H)$ (set of all bounded operators on a Hilbert space $H$) which is given by,
		\begin{equation*}
			Ad_{\delta}(T)=\tilde{\delta} (T\otimes 1_{\mathcal{A}}) \tilde{\delta}^*\quad\text{where}\quad T\in B(H).
		\end{equation*}
		The map $Ad_{\delta}$ will be referred as the ``co-action implemented by a unitary co-representation $\delta$".
	\end{prop}
	\subsection{Quantum automorphisms of single edged (weighted or non-weighted) graphs:}\label{Qaut_simp_graph}
	There are two different existing notions of quantum symmetry in a simple graph, one was introduced by Bichon  (see \cite{Bichon2003}) and the other was introduced by Banica (\cite{Banica2005},\cite{Banica2007}). Before going to that, we recall the notion of quantum permutation group from \cite{Wang1998}.\par  
	Let $X$ be a finite set. For $i\in X$, let us denote the characteristic function on $i$ as $\chi_i$, that is, $\chi_i(j)=\delta_{i,j}$ ($\delta_{i,j}=1$ if $i=j$ and $0$ otherwise) for all $j\in X$. The function algebra on $X$, that is, set of all functions from $X$ to $\mathbb{C}$,  is the $\mathbb{C}$-linear span of the elements $\{\chi_i|i\in X\}$. This function algebra will be treated as both an algebra (with multiplication given by, $\chi_i.\chi_j=\delta_{i,j}$)  and a Hilbert space (with inner product given by, $<\chi_i,\chi_j>=\delta_{i,j}$). We will denote this function algebra by $C(X)$ when we will treat it as an algebra and $L^2(X)$ when we will treat it as a Hilbert space.    
	\begin{defn}\label{quan_perm_def}
		Let $X_n=\{1,2,..,n\}$ be a finite set. The \textbf{quantum permutation group}  on $n$ elements, $S^+_n$ is the universal C* algebra generated by the elements of the matrix $(x_{ij})_{i,j=1,..,n}$ satisfying the following relations:
		\begin{enumerate}
			\item ${x^2_{ij}}=x_{ij}={x_{ij}}^*$ for all $i,j=1,..,n$.
			\item $\sum_{i=1}^nx_{ij}=1=\sum_{i=1}^n{x_{ji}}$ for all $j=1,..,n$.
		\end{enumerate}
		The co-product $\Delta_n$ on  $S^+_n$ is given by
		$\Delta_n(x_{ij})=\sum_{k=1}^{n}x_{ik}\otimes x_{kj}.$ 
	\end{defn}
	
	The relations (1) and (2) listed in definition \ref{quan_perm_def} will be referred to as \textbf{quantum permutation relations}. The quantum permutation group $S^+_n$ is the universal object in the category of all compact quantum groups co-acting on $C(X_n)$.
	
	We introduce a notation which is standard in this context: 
	
	\begin{nota}\label{defn_alpha^2}

		Let $X$ be a finite set and $\alpha:C(X)\rightarrow C(X)\otimes \mathcal{A}$ be a co-action of a compact quantum group $(\mathcal{A},\Delta)$ with co-representation matrix $Q=(q_{ij})_{i,j\in V}$. Then we define $\alpha^{(2)}=(id\otimes id\otimes m)(id\otimes \Sigma_{23}\otimes id)(\alpha\otimes \alpha)$ where $m$ is the multiplication map in $\mathcal{A}$ and $\Sigma_{23}$ is the standard flip map on 2nd and 3rd coordinates of the tensor product. For $i,j,k,l\in V$, we observe that,
		\begin{equation*}
			\alpha^{(2)}(\chi_k\otimes\chi_l)=\sum_{i,j\in V}\chi_i\otimes\chi_j\otimes q_{ik}q_{jl}.
		\end{equation*}
		It is easy to check using quantum permutation relations that $\alpha^{(2)}$  is actually a unitary co-representation of $(\mathcal{A},\Delta)$ on the Hilbert space $L^2(X)\otimes L^2(X)$. Its contragradient co-representation $\overline{\alpha^{(2)}}$ is also unitary.
	\end{nota}
We recall theorem 2.2 from \cite{Banica2005}.
\begin{thm}\label{wt_preserve}
	Let $\alpha$ be a co-action of a CQG $(\mathcal{A},\Delta)$ on $C(X_n)$ with co-representation matrix $Q=(q_{ij})_{i,j=1,..,n}$ and $W\in M_n(\mathbb{C})$ be a complex valued $n\times n$ matrix. Let us write $W=\sum_{c\in \mathbb{C}}W^c$ where $W^c_{ij}=1$ iff $W_{ij}=c$ and $W^c_{ij}=0$ otherwise. For $c\in\mathbb{C}$, we consider the linear subspace $K^c$ of $L^2(X)\otimes L^2(X)$ defined by
	\begin{equation*}
	K^c=\text{linear span}\{\chi_k\otimes \chi_l|W_{kl}=c;\:\:k,l\in V\}.
	\end{equation*}
	Then the following conditions are equivalent:
	\begin{enumerate}
		\item $QW=WQ$.
		\item $\alpha^{(2)}(K^c)\subseteq K^c\otimes \mathcal{A}$ for all $c\in\mathbb{C}$.
		\item $QW^c=W^cQ$ for all $c\in\mathbb{C}$.
	\end{enumerate}
\end{thm}
Now we recall the notions of quantum symmetry in a single edged graph given by  Bichon (\cite{Bichon2003}) and Banica (\cite{Banica2005}).
\begin{defn}\label{maindef_simp}
	Let $(V,E,w)$ be a weighted single edged graph with its adjacency matrix $W$. A co-action $\alpha:C(V)\otimes \mathcal{A}$ of a CQG $(\mathcal{A},\Delta)$ on $C(V)$ is said to preserve \textbf{quantum symmetry of $(V,E,w)$ in Banica's sense} if any of the following equivalent statement holds:
	\begin{enumerate}
		\item $QW=WQ$ where $Q$ is the co-representation matrix of $\alpha$.
		\item For all $c\in\mathbb{C}$, $\alpha^{(2)}(L^2(E^c))\subseteq L^2(E^c)\otimes \mathcal{A}$ where $$E^c=\{(i,j)\in E\:|\: w((i,j))=c\}.$$ 
\end{enumerate}
Moreover, $\alpha$ is said to preserve \textbf{quantum symmetry of $(V,E)$ in Bichon's sense} if $\alpha^{(2)}$ is a co-action on the algebra $C(E)$.
\end{defn}
The categories $\mathcal{D}^{Ban}_{(V,E,w)}$ and $\mathcal{D}^{Bic}_{(V,E,w)}$ consisting of CQGs co-acting on $(V,E,w)$ preserving its quantum symmetry in Banica's sense of Bichon's sense respectively admit universal objects namely $S^{Ban}_{(V,E,w)}$ and $S^{Bic}_{(V,E,w)}$. These are two different \textbf{quantum automorphism groups} of $(V,E,w)$.  
	
\subsection{Free wreath product by quantum permutation groups:}
	We recall the construction of \textbf{free wreath product by quantum permutation groups} formulated by Bichon in \cite{Bichon2004}. Similar treatment also works if we consider any subgroup of a quantum permutation groups. (see \cite{Banica2007a}).
	\par  
	Let $(\mathcal{B},\Delta')$ be a quantum subgroup of $S^+_n$ where  $S_n^+$ is the quantum permutation group on $n$ elements. Let $(\mathcal{A},\Delta)$ be another compact quantum group.  We consider $\mathcal{A}^{*n}$ to be $n$ times free product of the C* algebra $\mathcal{A}$ with the canonical inclusion maps $\nu_i:\mathcal{A}\rightarrow \mathcal{A}^{*n}$ where $i=1,2,..,n$. The algebra $\mathcal{A}^{*n}$ has a natural co-product structure coming from $(\mathcal{A},\Delta)$ making it a compact quantum group (see \cite{Wang1995}). We observe that, there is a natural co-action $\alpha:\mathcal{A}^{*n}\rightarrow\mathcal{A}^{*n}\otimes \mathcal{B}$ of the CQG $(\mathcal{B},\Delta')$ on the algebra $\mathcal{A}^{*n}$, which is given by,
	\begin{equation}\label{wr_pft_def_cann_action}
		\alpha(\nu_i(a))=\sum_{j=1}^n \nu_j(a)\otimes x_{ji}\quad\text{where}\quad i=1,2,..,n\quad\text{and}\quad a\in\mathcal{A}.  
	\end{equation}
	Here $(x_{ij})_{i,j=1,..,n}$ is the matrix of canonical generators of $\mathcal{B}$  satisfying \textbf{quantum permutation relations}.
	\begin{defn}
		 The \textbf{free wreath product of $(\mathcal{A},\Delta)$ by $(\mathcal{B},\Delta')$} is the quotient of the C* algebra $\mathcal{A}^{*n}*\mathcal{B}$ by a two sided C* ideal generated by the elements:
		\begin{equation*}
			\nu_i(a)x_{ij}-x_{ij}\nu_i(a),\quad 1\leq i,j\leq n,\quad a\in \mathcal{A}.
		\end{equation*}
		The free wreath product of $(\mathcal{A},\Delta)$ by $(\mathcal{B},\Delta')$ will be denoted by $\mathcal{A}*_w \mathcal{B}$.
	\end{defn}
	We recall theorem 3.2 from \cite{Bichon2004} which describes the  co-product structure on $\mathcal{A}*_w \mathcal{B}$.
	\begin{thm}\label{copdt_wr_pdt}
		There  is a natural co-product structure $\Delta_w$ on $\mathcal{A}*_w \mathcal{B}$ making it a compact quantum group. The co-product $\Delta_w$ satisfies:
		\begin{equation*}
			\Delta_w(x_{ij})=\sum_{k=1}^n x_{ik}\otimes x_{kj},\quad 
			\Delta_w(\nu_i(a))=\sum_{k=1}^n \nu_i\otimes \nu_k(\Delta(a))(x_{ik}\otimes 1).
		\end{equation*}	
		for all $i,j=1,..,n$ and $a\in \mathcal{A}$.
	\end{thm}
	As an immidiate application  of the above construction in the theory of quantum symmetry in simple graphs we state theorem 4.2 from \cite{Bichon2004} (see also theorem 7.1 from \cite{Banica2007a}).
	\begin{thm}\label{disjoint_union_simp_graph_qaut}
		Let $(V,E)$ be a finite connected simple graph (without loops). Let us consider another simple graph $(V^n, E^n)$ which is the disjoint union of $n$ copies of $(V,E)$. We have the following isomorphisms:
		\begin{align*}
			S^{Bic}_{(V^n,E^n)}&\cong S^{Bic}_{(V,E)}\:*_w\: S^+_n,\\ 
			S^{Ban}_{(V^n,E^n)}&\cong S^{Ban}_{(V,E)}\:*_w\: S^+_n
		\end{align*}
		where the underlying co-action of $S^+_n$ is given in equation \ref{wr_pft_def_cann_action}.
	\end{thm}

	\subsection{Setup and Notations:}\label{setup}
	 We introduce some notations and conventions that we will use throughout the rest of this article.
	Let $(V,E)$ be a multigraph with source and target maps $s:E\rightarrow V$ and $t:E\rightarrow V$. \textbf{We further assume that there is no isolated vertex, that is, every vertex is either an initial or final vertex of some edge}. 
	
	\begin{enumerate}
		\item For $i,j\in V$, we denote the the subsets $E^i$, $E_j$ and $E^i_j$ of $E$ by the following descriptions:
		\begin{align*}
			E^i_j&:=\{\tau\in E|s(\tau)=i\:\:\text{and}\:\: t(\tau)=j\};\\
			E^i&:=\{\tau\in E|s(\tau)=i\};\qquad E_j:=\{\tau\in E|t(\tau)=j\}.
		\end{align*}
         \item Let us define the sets of initial and final vertices $V^s\subseteq V$ and $V^t\subseteq V$ by
		\begin{equation*}
			V^s=s(E)\quad\text{and}\quad V^t=t(E). 
		\end{equation*}
		As our graphs do not have any isolated vertex, it is clear that $V=V^s\cup V^t$.
		\item\label{bimodule str} There is a natural $C(V^s)-C(V^t)$ bimodule structure on $L^2(E)$ which is given by
		\begin{equation}\label{bimod_str_def}
			\chi_i.\chi_{\tau}=\delta_{i,s(\tau)}\chi_{\tau}\quad
			\text{and}\quad\chi_{\tau}.\chi_j=\delta_{t(\tau),j}\chi_{\tau}
		\end{equation}
		where $i\in V^s,j\in V^t$ and $\tau\in E$. The Hilbert space $L^2(E)$ can also be treated as a $C(V)-C(V)$ bimodule with the same left and right module multiplication maps given by equations \ref{bimod_str_def}.
		
		\item For $\tau\in E$, let $p_{\tau}$ denote the orthogonal projection onto a subspace generated by $\chi_{\tau}$ in $L^2(E)$. We define two injective algebra maps $S:C(V^s)\rightarrow B(L^2(E))$ and $T:C(V^t)\rightarrow B(L^2(E))$ by
		\begin{equation*}
			S(\chi_v)=\sum_{\tau\in E^v} p_{\tau} \quad\text{and}\quad
			T(\chi_w)=\sum_{\tau\in E_w}p_{\tau}
		\end{equation*}
		for all $v$ in $V^s$ and $w$ in $V^t$.
		
		\item For $i,j\in V$ with $E^i_j\neq\phi$, let $p_{ij}$ be the orthogonal projection onto a linear subspace in $L^2(E)$  generated by the elements $\{\chi_{\tau}|\tau\in E^i_j\}$.  Let us define the following subalgebras in $B(L^2(E))$ by
		\begin{equation*}
			M_{ij}:=p_{ij}B(L^2(E))p_{ij}\quad\text{and}\quad D_{ij}:=p_{ij}Dp_{ij} 
		\end{equation*}
		where $D$ is the algebra of diagonal operators spanned by the elements $\{p_{\tau}|\tau\in E\}$.
		
	\end{enumerate}
	
\section{\textbf{Quantum symmetry in multigraphs}}\label{chap_quan_sym_direc_mult}
\par
Let us fix a multigraph $(V,E)$ with source and target maps $s:E\rightarrow V$ and $t:E\rightarrow V$ and no isolated vertex.  

\begin{defn}
	By a ``bi-unitary" co-representation $\beta$ of a CQG $(\mathcal{A},\Delta)$ on a finite dimensional Hilbert space,  we mean a unitary co-representation $\beta$ such that its contragradient $\overline{\beta}$ is also unitary. 
\end{defn} 

We make some observations before moving to the main results.
\begin{lem}\label{charac}
	Let $F\in B(L^2(E))$. Then,
	\begin{enumerate}
		\item $F\in S(C(V^s))$ if and only if the following holds:
		
		For $\tau,\tau_1,\tau_2\in E$,	$F(\chi_{\tau})=c_{\tau}\chi_{\tau}$ for some $c_{\tau}\in \mathbb{C}$ and $c_{\tau_1}= c_{\tau_2}$ whenever $s(\tau_1)=s(\tau_2)$.  
		
		\item $F\in T(C(V^t))$ if and only if the following holds:
		
		For $\tau,\tau_1,\tau_2\in E$,	$F(\chi_{\tau})=c_{\tau}\chi_{\tau}$ for some $c_{\tau}\in \mathbb{C}$ and $c_{\tau_1}= c_{\tau_2}$ whenever  $t(\tau_1)=t(\tau_2)$.  
	\end{enumerate}
\end{lem}

\begin{lem}\label{tech_lem1}
	Let $\{A_i|i=1,2,..,n\}$ be a finite set of operators on a Hilbert space and $p$ and $q$ be two projections.
	\begin{enumerate}
		\item If $\sum_{i=1}^{n}A_iA^*_i=p$, then $pA_i=A_i$.
		\item If $\sum_{i=1}^{n}A^*_iA_i=q$, then $A_iq=A_i$.
	\end{enumerate}
	
\end{lem}
\begin{proof}
	To prove (1), we observe that $A_iA^*_i\leq p$ for all $i$.
	It is enough to show that $\|(1-p)A_i\|=0$ for all $i$.
	\begin{equation*}
		\|(1-p)A_i\|^2=\|(1-p)A_iA^*_i(1-p)\|\leq \|(1-p)p(1-p)\|=0.
	\end{equation*}
	\par
	To prove (2), it is enough to observe that $qA^*_i=A^*_i$ which we get by replacing $A_i$ with $A^*_i$ in the first identity.
\end{proof}
\subsection{Left and right equivariant co-representations on $L^2(E)$:}\label{left_equiv_co-rep}
Seeing $L^2(E)$ as a left $C(V^s)$ module we formulate an equivalent criterion for left equivariant bi-unitary co-representations on $L^2(E)$. 
\begin{thm}\label{bi-mod_equiv_source}
	Let $\beta:L^2(E)\rightarrow L^2(E)\otimes \mathcal{A}$	be a bi-unitary co-representation of a CQG $(\mathcal{A},\Delta)$. Let $Ad_{\beta}$ be the co-action on $B(L^2(E))$ implemented by the unitary co-representation $\beta$ (see proposition \ref{adj_co-action}). Then the following conditions are equivalent:
	\begin{enumerate}
		\item $Ad_{\beta}(S(C(V^s)))\subseteq S(C(V^s))\otimes \mathcal{A}$.
		\item There exists a co-action $\alpha_s:C(V^s)\rightarrow C(V^s)\otimes \mathcal{A}$ such that,
		\begin{equation*}
			\alpha_s(\chi_i)\beta(\chi_{\tau})=\beta(\chi_i.\chi_{\tau})
		\end{equation*}
		for all $i\in V^s$ and $\tau\in E$.
	\end{enumerate}
\end{thm}
\begin{proof}
	Let $U=(u^{\sigma}_{\tau})_{\sigma,\tau\in E}$ be the co-representation matrix of $\beta$.\\ 
	We make some observations first before proving the equivalence.
	Let us fix $k\in V^s$ and $\sigma_2\in E$. We observe that,
	\begin{align*}
		Ad_{\beta}(S(\chi_k))(\chi_{\sigma_2}\otimes 1)&=Ad_{\beta}(\sum_{\tau\in E^k}p_{\tau})(\chi_{\sigma_2}\otimes 1)\\
		&=\beta(\sum_{\tau\in E^k}p_{\tau}\otimes 1)(\sum_{\tau'\in E}\chi_{\tau'}\otimes {u^{\sigma_2}_{\tau'}}^*)\\
		&=\sum_{\sigma_1\in E}\chi_{\sigma_1}\otimes (\sum_{\tau\in E^k}u^{\sigma_1}_{\tau}{u^{\sigma_2}_{\tau}}^*)
	\end{align*}
	Applying lemma \ref{charac} we get that, for all $k\in V^s$ and $\sigma_1,\sigma_2\in E$
	\begin{align}
		Ad_{\beta}(S(C(V^s)))\subseteq S(C(V^s))\otimes \mathcal{A}\iff
		&\sum_{\tau\in E^k}u^{\sigma_1}_{\tau}{u^{\sigma_2}_{\tau}}^*=0\quad\text{if}\quad \sigma_1\neq\sigma_2\notag\\
		\text{and}\quad &\sum_{\tau\in E^k}{u^{\sigma_1}_{\tau}u^{\sigma_1}_\tau}^*=\sum_{\tau\in E^k}{u^{\sigma_2}_{\tau}u^{\sigma_2}_\tau}^*\quad\text{if}\quad
		s(\sigma_1)=s(\sigma_2)\label{source_iden_simp}.
	\end{align}
	Let $\alpha_s:C(V^s)\rightarrow C(V^s)\otimes\mathcal{A}$ be a co-action on $C(V^s)$ with co-representation matrix $(q^i_j)_{i,j\in V^s}$. Let $i\in V^s$ and $\tau\in E$. We observe that  
	\begin{equation}\label{bi-mod_charac_source_abs}
		\alpha_s(\chi_i)\beta(\chi_{\tau})=\beta(\chi_i.\chi_{\tau})\iff q^{s(\sigma)}_{i}u^{\sigma}_{\tau}=\delta_{i,s(\tau)}u^{\sigma}_{\tau}
	\end{equation}
	for all  $\sigma\in E$.\\
	\textbf{Claim:} $(1)\implies (2)$.\\
	From our assumption and observation \ref{source_iden_simp} it follows that,
	\begin{equation}\label{eigvec1_simp}
		Ad_{\beta}(S(\chi_k))(\chi_{\sigma_2}\otimes 1)=\chi_{\sigma_2}\otimes (\sum_{\tau\in E^k}u^{\sigma_2}_{\tau}{u^{\sigma_2}_{\tau}}^*)\quad\text{for all}\quad k\in V^s,\sigma_2\in E.
	\end{equation}
	For $k\in V^s$ and $\sigma_2\in E$, let us define
	\begin{equation}\label{source_perm_simp}
		\sum_{\tau\in E^k}u^{\sigma_2}_{\tau}{u^{\sigma_2}_{\tau}}^*=q^{s(\sigma_2)}_k.
	\end{equation}
	From equation \ref{eigvec1_simp} we further observe that,
	\begin{align*}
		Ad_{\beta}(S(\chi_k))=\sum_{\sigma\in E}p_{\sigma}\otimes q^{s(\sigma)}_k
		=\sum_{i\in V^s}(\sum_{\sigma\in E^{i}}p_{\sigma})\otimes q^{i}_{k}
		=\sum_{i\in V^s} S(\chi_{i})\otimes q^{i}_k.
	\end{align*}
	As $C(V^s)\cong S(C(V^s))$ as algebras and $Ad_{\beta}$ is already a co-action on $S(C(V^s))$, we define a \textbf{quantum permutation} $\alpha_s:C(V^s)\rightarrow C(V^s)\otimes \mathcal{A}$ by the following expression:
	\begin{equation*}
		\alpha_s(\chi_k)=\sum_{i\in V^s}\chi_i\otimes q^i_k\quad \text{for all}\quad k\in V^s.
	\end{equation*}
	Let us now fix $i\in V^s$ and $\sigma,\tau\in E$. From equation \ref{source_perm_simp} and lemma \ref{tech_lem1} it follows that,
	\begin{equation*}
		q^{s(\sigma)}_iu^{\sigma}_{\tau}=q^{s(\sigma)}_i(q^{s(\sigma)}_{s(\tau)}u^{\sigma}_{\tau})=\delta_{i,s(\tau)}u^{\sigma}_{\tau}.
	\end{equation*}
	Using observation \ref{bi-mod_charac_source_abs} we conclude that (2) follows.
	\\
	\\
	\textbf{Claim:}$(2)\implies (1)$.\\
	Let $(q^i_j)_{i,j\in V^s}$ and $(u^{\sigma}_{\tau})_{\sigma,\tau\in E}$ be co-representation matrices of $\alpha_s$ and $\beta$.\\
	Let $\sigma_1,\sigma_2\in E$ and $k\in V^s$. As $\beta$ is unitary, using observation \ref{bi-mod_charac_source_abs} it follows that,
	\begin{equation*}
		\sum_{\tau\in E^k}u^{\sigma_1}_{\tau}{u^{\sigma_2}_{\tau}}^*=q^{s(\sigma_1)}_k(\sum_{\tau\in E}u^{\sigma_1}_{\tau}{u^{\sigma_2}_{\tau}}^*)q^{s(\sigma_2)}_k=\delta_{\sigma_1,\sigma_2}q^{s(\sigma_1)}_kq^{s(\sigma_2)}_k.
	\end{equation*}
	Hence we get,
	\begin{align*}
		\sum_{\tau\in E^k}u^{\sigma_1}_{\tau}{u^{\sigma_2}_{\tau}}^*&=0\quad\text{if}\quad \sigma_1\neq\sigma_2\notag\\
		\text{and}\quad \sum_{\tau\in E^k}{u^{\sigma_1}_{\tau}u^{\sigma_1}_\tau}^*&=\sum_{\tau\in E^k}{u^{\sigma_2}_{\tau}u^{\sigma_2}_\tau}^*\quad\text{if}\quad
		s(\sigma_1)=s(\sigma_2).
	\end{align*}
	Therefore (1) follows from observation \ref{source_iden_simp}.
\end{proof}

Seeing $L^2(E)$ as a right $C(V^t)$ module we formulate a equivalent criterion for right equivariant bi-unitary  co-representations on $L^2(E)$. 

\begin{thm}\label{bi-mod_equiv_targ}
	Let $\beta:L^2(E)\rightarrow L^2(E)\otimes \mathcal{A}$	be a bi-unitary co-representation of a CQG $(\mathcal{A},\Delta)$. Let us consider the co-action $Ad_{\overline{\beta}}$ on $B(L^2(E))$ implemented by the unitary co-representation $\overline{\beta}$. The following conditions are equivalent:
	\begin{enumerate}
		\item $Ad_{\overline{\beta}}(T(C(V^t)))\subseteq T(C(V^t))\otimes \mathcal{A}$.
		\item There exists a co-action $\alpha_t:C(V^t)\rightarrow C(V^t)\otimes \mathcal{A}$ such that,
		\begin{equation*}
			\beta(\chi_{\tau})\alpha_t(\chi_j)=\beta(\chi_{\tau}.\chi_j)
		\end{equation*}
		for all $j\in V^t$ and $\tau\in E$. 
	\end{enumerate}
\end{thm}
\begin{proof}
Using similar arguments as in the proof of theorem \ref{bi-mod_equiv_source} for unitary co-representation $\overline{\beta}$, theorem \ref{bi-mod_equiv_targ} follows. 
\end{proof}
\subsubsection{\textbf{Induced permutations on $V^s$ and $V^t$:}}\label{alpha-s_alpha-t_rem_simp}

It is clear that  the co-actions  $\alpha_s$ and $\alpha_t$ satisfying (2) in theorem \ref{bi-mod_equiv_source}
and theorem \ref{bi-mod_equiv_targ} are essentially unique as they are completely determined by the bi-unitary co-representation $\beta$. Given a bi-unitary co-representation $\beta$ satisfying (1) in theorem \ref{bi-mod_equiv_source} and theorem \ref{bi-mod_equiv_targ}, we will refer $\alpha_s$ and $\alpha_t$ as \textbf{induced co-actions on $C(V^s)$ and $C(V^t)$}. In terms of coefficients of  co-representation matrices, we have the following:
\begin{align}
\sum_{\tau\in  E^k}u^{\sigma_1}_{\tau}u^{\sigma_2*}_{\tau}=\delta_{\sigma_1,\sigma_2}q^{s(\sigma_1)}_k,\quad&\text{and}\quad
\sum_{\tau\in E_l}u^{\sigma_1*}_{\tau}u^{\sigma_2}_{\tau}=\delta_{\sigma_1,\sigma_2}r^{t(\sigma_1)}_l;\label{u_q_relations}\\
q^{s(\sigma)}_{i}u^{\sigma}_{\tau}=\delta_{i,s(\tau)}u^{\sigma}_{\tau}\quad
&\text{and}\quad u^{\sigma}_{\tau}r^{t(\sigma)}_{j}=\delta_{t(\tau),j}u^{\sigma}_{\tau}\label{bimod_coeff}
\end{align}  
where $(u^{\sigma}_{\tau})_{\sigma,\tau\in E}$, $(q^i_k)_{i,k\in V^s}$, $(r^j_l)_{j,l\in V^t}$ are co-representation matrices of $\beta$, $\alpha_s$ and $\alpha_t$.

\subsection{Induced permutations on $V^s\cap V^t$:}\label{induced_perm}
It is not enough  to only consider $C(V^s)-L^2(E)-C(V^t)$ bimodule structure as there does exist non-isomorphic single edged graphs which have non-isomorphic quantum automorphism groups but isomorphic $C(V^s)-L^2(E)-C(V^t)$ bimodule structure. See the graphs in figure \ref{Two graphs with isomorphic $C(V^s)-L^2(E)-C(V^t)$ bimodule structure.} for example, where the left one does not have any quantum symmetry (in Banica's sense) but the right one does have. 
\begin{figure}[h]
	\centering
	\begin{tikzpicture}
		\node at (-1,0) {$\bullet$};
		\node at (1,0) {$\bullet$};
		\draw[thick, out=45, in=135, looseness=1.2] (-1,0) to node{\midarrow} (1,0);
		\draw[thick, out=-45, in=-135, looseness=1.2] (-1,0) to node{\backmidarrow} (1,0);
		\draw[thick, out=135, in=90, looseness=1.2] (-1,0)   to (-2.5, 0);
		\draw[thick, out=-90, in=-135, looseness=1.2] (-2.5,0) to node{\midarrow} (-1,0);
		\draw[thick, out=45, in=90, looseness=1.2] (1,0) to node{\midarrow} (2.5,0);
		\draw[thick, out=-90, in=-45, looseness=1.2] (2.5,0) to (1,0);
		\node at (-1,-.5) {$1$};
		\node at (1,-.5) {$2$};
	\end{tikzpicture}
	\quad\quad\quad\quad
	\begin{tikzpicture}
		\node at (-1,-1) {$\bullet$};
		\node at (1,-1) {$\bullet$};
		\node at (1,1) {$\bullet$};
		\node at (-1,1) {$\bullet$};
		\draw[thick] (-1,-1) to node{\midarrow} (1,-1);
		\draw[thick] (-1,-1) edge[bend left] node{\slantmidarrow}  (1,1);
		\draw[thick](-1,1) edge[bend right] node{\opslantmidarrow} (1,-1);
		\draw[thick] (-1,1) to node{\midarrow} (1,1);
		\node at (-1.5,-1) {$2$};
		\node at (1.5,-1) {$4$};
		\node at (1.5,1) {$3$};
		\node at (-1.5,1) {$1$};
	\end{tikzpicture}
	\caption{Two graphs with isomorphic $C(V^s)-L^2(E)-C(V^t)$ bimodule structure.}
	\label{Two graphs with isomorphic $C(V^s)-L^2(E)-C(V^t)$ bimodule structure.}
\end{figure}
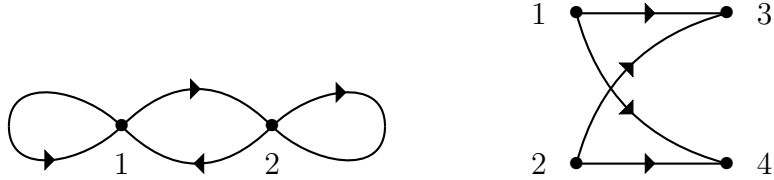
Continuing our investigations further, we found that it is important to consider right equivariance of $\alpha_s$ and left equivariance of $\alpha_t$ on the set of edges with at least one of their endpoints in $V^s\cap V^t$. We propose the following result:

\begin{thm}\label{source_targ_consis_thm}
	Let $\beta:L^2(E)\rightarrow L^2(E)\otimes \mathcal{A}$ be a bi-unitary co-representation of a CQG $(\mathcal{A},\Delta)$ such that the following conditions  hold:
	\begin{enumerate}
		\item $Ad_{\beta}(S(C(V^s)))\subseteq S(C(V^s))\otimes \mathcal{A}$.
		\item $Ad_{\overline{\beta}}(T(C(V^t)))\subseteq T(C(V^t))\otimes \mathcal{A}$.	
	\end{enumerate}
	Furthermore, we assume that the induced co-actions  $\alpha_s$ and $\alpha_t$ (see subsection \ref{alpha-s_alpha-t_rem_simp}) both preserve $C(V^s\cap V^t)$, that is, 
	\begin{align*}
		\alpha_s(C(V^s \cap V^t))&\subseteq C(V^s\cap   V^t)\otimes\mathcal{A}\subseteq C(V^s)\otimes \mathcal{A},\\  
		\alpha_t(C(V^s\cap V^t))&\subseteq C(V^s\cap V^t)\otimes\mathcal{A}\subseteq C(V^t)\otimes \mathcal{A}.
	\end{align*} 
	Then the following conditions are equivalent:
	\begin{enumerate}
		\item $\alpha_s|_{C(V^s\cap V^t)}=\alpha_t|_{C(V^s\cap V^t)}$.
		\item For all $j\in V^s\cap V^t$ and $\tau\in E$,
		\begin{equation*}
			\beta(\chi_\tau)\alpha_s(\chi_j)=\beta(\chi_\tau .\chi_j).
		\end{equation*}
		\item For all $i\in V^s\cap V^t$ and $\tau\in E$,
		\begin{equation*}
			\alpha_t(\chi_i)\beta(\chi_{\tau})=\beta(\chi_i.\chi_\tau).
		\end{equation*}
	\end{enumerate}
\end{thm} 
\begin{proof}
	
	We define $E_{V^s},E^{V^t}\subseteq E$  by 
	\begin{align*}
		E_{V^s}=\{\tau\in E|t(\tau)\in V^s\cap V^t\}\quad\text{and}\quad
		E^{V^t}=\{\tau\in E|s(\tau)\in V^s\cap V^t\}.
	\end{align*} 
	We make some observations first.
	Let $(u^{\sigma}_{\tau})_{\sigma,\tau\in E}$, $(q^l_j)_{l,j\in V^s}$ and $(r^k_i)_{k,i\in V^t}$ be co-representation matrices of $\beta$, $\alpha_s$ and $\alpha_t$ respectively. 
	For $\tau\in E$ and $j\in V^s\cap V^t$ we observe that,
	\begin{align*}
		\beta(\chi_{\tau})\alpha_s(j)=(\sum_{\sigma\in E}\chi_{\sigma}\otimes u^{\sigma}_{\tau})(\sum_{l\in V^s\cap V^t}\chi_l\otimes q^l_j)
		=\sum_{\sigma\in E_{V^s}}\chi_{\sigma}\otimes u^{\sigma}_{\tau}q^{t(\sigma)}_j.
	\end{align*}
	Hence for all $j\in V^s\cap V^t$ and $\tau\in E$, $$\beta(\chi_{\tau})\alpha_s(j)=\beta(\chi_\tau . \chi_j)$$ if and only if 
	\begin{align}
		&\beta(L^2(E_{V^s}))\subseteq L^2(E_{V^s}) \otimes \mathcal{A}\quad\text{and}\notag\\
		&u^{\sigma}_{\tau}q^{t(\sigma)}_j=\delta_{t(\tau),j}u^{\sigma}_{\tau}\quad\text{whenever}\quad \sigma\in E_{V^s}.\label{restr_bi-mod_targ} 
	\end{align}
	Similarly it also follows that, for all $i\in V^s\cap V^t$ and $\tau\in E$ $$\alpha_t(\chi_{i})\beta(\chi_\tau)=\beta(\chi_i.\chi_\tau)$$ if and only if 
	\begin{align}
		&\beta(L^2(E^{V^t}))\subseteq L^2(E^{V^t})\otimes \mathcal{A}\quad\text{and}\notag\\
		& r^{s(\sigma)}_iu^\sigma_\tau=\delta_{i,s(\tau)}u^\sigma_\tau\quad\text{whenever}\quad \sigma\in E^{V^t}\label{restr_bi-mod_source}.
	\end{align}
	Now we proceed to prove our theorem.\\
	\textbf{Claim:}$(1)\implies (2)$;$(1)\implies (3)$\\
	As  $\alpha_s|_{C(V^s\cap V^t)}=\alpha_t|_{C(V^s\cap V^t)}$, for $i\in V^s\cap V^t$ we have,
	\begin{align*}
		q^k_i&=r^k_i\quad\text{when}\quad k\in V^s\cap V^t\quad\text{and}\\
		q^k_i&=0=r^l_i\quad\text{when}\quad k\in V^s\setminus  V^t, l\in V^t\setminus V^s. 
	\end{align*}
	From above expressions, theorem \ref{bi-mod_equiv_targ} and theorem \ref{bi-mod_equiv_source} it follows that, for $\sigma,\tau\in E$, 
	\begin{align*}
		u^{\sigma}_{\tau}&=u^{\sigma}_{\tau}r^{t(\sigma)}_{t(\tau)}=0\quad\text{whenever}\quad \sigma\notin E_{V^s}\quad \text{but}\quad\tau\in E_{V^s}\\ \text{and}\quad
		u^{\sigma}_{\tau}&=q^{s(\sigma)}_{s(\tau)}u^{\sigma}_{\tau}=0\quad\text{whenever}\quad \sigma\notin E^{V^t}\quad \text{but}\quad\tau\in E^{V^t}.
	\end{align*}
	Hence we have, 
	\begin{align*}
		\beta(L^2(E_{V^s}))\subseteq L^2(E_{V^s})\otimes \mathcal{A}\quad\text{and}\quad
		\beta(L^2(E^{V^t}))\subseteq L^2(E^{V^t})\otimes \mathcal{A}.
	\end{align*}
	We further observe that for $i,j\in V^s\cap V^t, \sigma_1\in E_{V^s}$ and $\sigma_2\in E^{V^t}$, 
	\begin{align*}
		u^{\sigma_1}_{\tau}q^{t(\sigma_1)}_j=u^{\sigma_1}_{\tau}r^{t(\sigma_1)}_j=\delta_{j,t(\tau)}u^{\sigma_1}_{\tau}\quad 
		\text{and}\quad r^{s(\sigma_2)}_iu^{\sigma_2}_{\tau}=q^{s(\sigma_2)}_iu^{\sigma_2}_{\tau}=\delta_{i,s(\tau)}u^{\sigma_2}_{\tau}.
	\end{align*}
	As our choice of $i,j,\sigma_1,\sigma_2$ was arbitrary, from observations   \ref{restr_bi-mod_targ} and \ref{restr_bi-mod_source}, (2) and (3) follow.\\
	\textbf{Claim:}$(2)\implies (1)$.\\
	Let $i,k\in V^s\cap V^t$ and $\sigma\in E$ be such that $t(\sigma)=k$. Using equations \ref{u_q_relations} and \ref{restr_bi-mod_targ} we observe that, 
	\begin{equation}\label{source_targ_perm_consis_4}
		r^k_i=\sum_{\tau\in E_i}{u^{\sigma}_{\tau}}^*u^{\sigma}_{\tau}=\sum_{\tau\in E_i}{u^{\sigma}_{\tau}}^*u^{\sigma}_{\tau}q^k_i=r^k_iq^k_i
	\end{equation}
	Hence it follows that,
	\begin{equation}
		r^k_i\leq q^k_i \quad\text{for all}\quad i,k\in V^s\cap V^t.
	\end{equation} 
	As coefficients of both matrices  $(q^k_i)_{k,i\in V^s\cap V^t}$ and $(r^k_i)_{k,i\in V^s\cap V^t}$ satisfy quantum permutation relations it follows that, for $i\in V^s\cap V^t$,
	
	\begin{align*}
		1=&\sum_{k\in V^s\cap V^t}r^k_i\leq \sum_{k\in V^s\cap V^t}q^k_i=1
	\end{align*}
	As $\{r^k_i\:|\:k\in V^s\cap V^t\}$ and $\{q^k_i\:|\:k\in V^s\cap V^t\}$ both are sets of mutually orthogonal projections, we have $$q^k_i=r^k_i\quad  \text{for all}\quad  i,k\in V^s\cap V^t.$$
	Therefore (1) follows.\\
	\textbf{Claim:}$(3)\implies (1)$\\
	Let $i,k\in V^s\cap V^t$ and $\sigma\in E$ be such that $s(\sigma)=k$. Using equations \ref{u_q_relations} and \ref{restr_bi-mod_source} we observe that,
	\begin{equation*}
		q^k_i=\sum_{\tau\in E^i}u^{\sigma}_{\tau}u^{\sigma*}_{\tau}=r^k_i(\sum_{\tau\in E^i}u^{\sigma}_{\tau}u^{\sigma*}_{\tau})=r^k_iq^k_i.
	\end{equation*}
	hence it follows that,
	\begin{equation*}
		q^k_i\leq r^k_i\quad\text{for all}\quad i,k\in V^s\cap V^t.
	\end{equation*}
	Using similar arguments used in the previous case, (1) follows.
\end{proof}
\subsection{Co-actions on a multigraph}\label{equiv_def_quan_sym_simp}
In light of above discussions, we introduce the notion of quantum symmetry preserving co-action on a multigraph formulated in terms of  bi-unitary maps. 
\begin{defn}\label{maindef}
	A compact quantum group $(\mathcal{A},\Delta)$ is said to co-act on a multigraph $(V,E)$ \textbf{preserving its quantum symmetry in Banica's sense} if there exists a bi-unitary co-representation $\beta:L^2(E)\rightarrow L^2(E)\otimes \mathcal{A}$ such that the following conditions hold:
	\begin{enumerate}
		\item $Ad_{\beta}(S(C(V^s)))\subseteq S(C(V^s))\otimes \mathcal{A}$.
		\item $Ad_{\overline{\beta}}(T(C(V^t)))\subseteq T(C(V^t))\otimes \mathcal{A}$.
		\item The induced co-actions $\alpha_s$ and $\alpha_t$ (see remark \ref{alpha-s_alpha-t_rem_simp}) preserve and agree on $C(V^s\cap V^t)$, that is,
		\begin{equation*}
			\alpha_s|_{C(V^s\cap V^t)}=\alpha_t|_{C(V^s\cap V^t)}.
		\end{equation*}
		\item $\beta$ fixes the element  $\xi_0:=\sum_{\tau\in E}\chi_{\tau}\in L^2(E)$, that is,
		\begin{equation*}
			\beta(\xi_0)=\xi_0\otimes 1_{\mathcal{A}}.
		\end{equation*}
	\end{enumerate}
	Moreover, if $\beta$ is a co-action on the algebra $C(E)$, that is, a \textbf{quantum permutation} of the edge set $E$, then we say that $(\mathcal{A},\Delta)$ \textbf{co-acts on $(V,E)$ preserving its quantum symmetry in Bichon's sense}.
\end{defn}
\begin{rem}\label{ver_quan_per}
	The reason we have used names of Banica and Bichon in definition \ref{maindef} is because in the context of single edged graphs, above definition is an equivalent description of definition \ref{maindef_simp} (see remark \ref{ban_iden1_rem_bic} and proposition \ref{banica_iden2})\par  
	 From condition (3) of definition \ref{maindef} it follows that the co-actions $\alpha_s$ and $\alpha_t$ induce a co-action $\alpha: C(V)\rightarrow C(V)\otimes \mathcal{A}$ by  
	\begin{align*}
	\alpha(\chi_k)&=\alpha_s(\chi_k)\quad\text{if}\quad k\in V^s,\\
	&=\alpha_t(\chi_k)\quad\text{if}\quad k\in V^t.			
	\end{align*}
	$\alpha$ is the required ``permutation" of vertices derived from ``permutation" of edges $E$ and will be referred as \textbf{induced permutation on the set of vertices of $(V,E)$}. In that case, as our next theorem states, we can now treat $L^2(E)$ as $C(V)-C(V)$ bimodule instead of  $C(V^s)-C(V^t)$ bimodule.
\end{rem}
\begin{thm}\label{bimod_prop_equiv}
	Let $\beta:L^2(E)\rightarrow L^2(E)\otimes \mathcal{A}$ be a bi-unitary co-representation and $\alpha:C(V)\rightarrow C(V)\otimes \mathcal{A}$ be a co-action of a CQG $(\mathcal{A},\Delta)$.  The following are equivalent:
	\begin{enumerate}
		\item $\alpha(\chi_i).\beta(\chi_\sigma)=\beta(\chi_i.\chi_\sigma)$ and $\beta(\chi_\sigma).\alpha(\chi_j)=\beta(\chi_\sigma.\chi_j)$ for all $i\in V^s,j\in V^t,\sigma\in E$.
		\item  $\alpha(\chi_i).\beta(\chi_\sigma)=\beta(\chi_i.\chi_\sigma)$ and $\beta(\chi_\sigma).\alpha(\chi_i)=\beta(\chi_\sigma.\chi_i)$ for all $i\in V,\sigma\in E$.
	\end{enumerate}
\end{thm}
\begin{proof}
	Let $(u^{\sigma}_{\tau})_{\sigma,\tau\in E}$ and $(q^i_j)_{i,j\in V}$ be the co-representation matrices of $\alpha$ and $\beta$.\par 
	$(1)\implies (2)$.\\
	Let us consider $\sigma\in E,i\in V^s$ and $j\in V^t$. As $\beta$ is bi-unitary, we observe that,
	\begin{align*}
		q^{s(\sigma)}_i&=q^{s(\sigma)}_i(\sum_{\tau\in E}u^{\sigma}_{\tau}u^{\sigma*}_{\tau})q^{s(\sigma)}_i=\sum_{\tau\in E^i}u^{\sigma}_{\tau}u^{\sigma*}_{\tau};\\
		q^{t(\sigma)}_{j}&=q^{t(\sigma)}_{j}(\sum_{\tau\in E}u^{\sigma*}_{\tau} u^{\sigma}_{\tau})q^{t(\sigma)}_{j}=\sum_{\tau\in E_j}u^{\sigma*}_{\tau} u^{\sigma}_{\tau}.
	\end{align*}
	We therefore have,
	\begin{equation*}
		\sum_{i\in V^s}q^{s(\sigma)}_i=\sum_{\tau\in E}u^{\sigma}_{\tau}u^{\sigma*}_{\tau}=1\quad\text{and}\quad 
		\sum_{j\in V^t}q^{t(\sigma)}_j=\sum_{\tau\in E}u^{\sigma*}_{\tau}u^{\sigma}_{\tau}=1.
	\end{equation*}
	Hence for $i\notin V^s$ and $j\notin V^t$,
	\begin{equation}\label{source_targ_coeff_presv}
		q^{s(\sigma)}_i=0\quad \text{and}\quad q^{t(\sigma)}_j=0.
	\end{equation}
	From our assumption and the observation made above, it follows that, for all $i\in V$ and $\sigma,\tau\in E$,
	\begin{align*}
		q^{s(\sigma)}_iu^{\sigma}_{\tau}&=\delta_{i,s(\tau)}u^{\sigma}_{\tau}\quad \text{when}\quad i\in V^s;\\
		&=0\\
		&=\delta_{i,s(\tau)}u^{\sigma}_{\tau}\quad\text{when}\quad i\notin V^s.
	\end{align*}
Similar identities hold for the target case. Therefore (2) is true. \\
	The converse $(2)\implies (1)$ is obvious. 
\end{proof}
From theorems \ref{bimod_prop_equiv}, \ref{bi-mod_equiv_source} and \ref{bi-mod_equiv_targ}, we have the following characterisation of Banica's notion of quantum symmetry (see definition \ref{maindef}).
\begin{cor}\label{bimod_prop_equiv_cor}
Let $\beta:L^2(E)\rightarrow L^2(E)\otimes \mathcal{A}$ be a bi-unitary co-representation of a CQG $(\mathcal{A},\Delta)$ on $L^2(E)$. Then $\beta$ preserves \textbf{quantum symmetry of $(V,E)$ in Banica's sense} if and only if the following conditions hold:  
\begin{enumerate}
\item $\beta(\xi_0)=\xi_0\otimes 1$ where $\xi_0=\sum_{\tau\in V}\chi_\tau$.
\item There exists an $\alpha:C(V)\rightarrow  C(V)\otimes\mathcal{A}$ such that for all $i\in V$ and $\tau\in E$,
    \begin{equation*}
        \alpha(\chi_i).\beta(\chi_\tau)=\beta(\chi_i.\chi_\tau)\quad\text{and}\quad\beta(\chi_\tau).\alpha(\chi_i)=\beta(\chi_\tau.\chi_i)
    \end{equation*}
\end{enumerate}

\end{cor}
\subsection{\textbf{Some useful identities:}} 
 The next two results show that our notions of quantum symmetry in multigraphs is consistent with the picture of quantum symmetry in single edged graphs.
\begin{prop}\label{banica_iden1}
	Let $\beta$ be a co-action of a CQG $(\mathcal{A},\Delta)$ on $(V,E)$ preserving its quantum symmetry in Banica's sense. Let $(u^{\sigma}_{\tau})_{\sigma,\tau\in E}$ and $(q^i_j)_{i,j\in V}$ be the co-representation matrices of $\beta$ and induced co-action $\alpha$ on $C(V)$. 
	For $i,j,k,l\in V$  with $E^i_j\neq \phi$ and $E^k_l\neq \phi$, we have the following:
	\begin{enumerate}
		\item For any $\sigma\in E^i,$ we have $\sum_{\tau\in E^k}u^{\sigma}_{\tau}=q^i_k$.
		\item For any $\sigma\in E_j$, we have $\sum_{\tau\in E_l}u^{\sigma}_{\tau}=q^j_l$.
		\item For any $\sigma\in E^i_j$, we have $\sum_{\tau\in E^k_l}u^{\sigma}_{\tau}=q^i_kq^j_l$.
	\end{enumerate}
	
	As $\beta:L^2(E)\rightarrow L^2(E)\otimes\mathcal{A}$ is bi-unitary co-representation, by using antipode on the underlying Hopf* algebra of matrix elements of $(\mathcal{A},\Delta)$, it follows that the above identities are true if we consider sum in upper indices instead of lower indices.
\end{prop}

\begin{proof}
 From (4) of definition \ref{maindef}  we have,
	
	\begin{equation*}
		\sum_{\sigma\in E}u^{\sigma}_{\tau}=1=\sum_{\sigma\in E}u^{\tau}_{\sigma}\quad \text{for all}\quad \tau\in E.
	\end{equation*}
	For $\sigma\in E^i$, using equation \ref{bimod_coeff} we observe that,
	\begin{equation*}
		q^i_k=q^i_k(\sum_{\tau\in E}u^{\sigma}_{\tau})
		=\sum_{\tau\in E^k}u^{\sigma}_{\tau}.
	\end{equation*}
	For $\sigma\in E_j$, using equation \ref{bimod_coeff} we observe that,
	\begin{equation*}
		q^j_l=(\sum_{\tau\in E}u^{\sigma}_{\tau})q^j_l
		=\sum_{\tau\in E_l}u^{\sigma}_{\tau}.
	\end{equation*}
	For $\sigma\in E^i_j$, using equation \ref{bimod_coeff} we note that,
	\begin{equation*}
		q^i_kq^j_l=q^i_k(\sum_{\tau\in E}u^{\sigma}_{\tau})q^j_l\\
		=\sum_{\tau\in E^k_l}u^{\sigma}_{\tau}.
	\end{equation*}
\end{proof}
\begin{rem}\label{ban_iden1_rem_bic}
  If $(V,E)$ is a single edged graph, using (3) of proposition \ref{banica_iden1} it follows that $\beta=\alpha^{(2)}$ where $\alpha$ is the induced permutation on the vertex set.\par 
   Moreover, if $\beta$ is a quantum permutation on the edge set $E$ that is, a co-action on $(V,E)$ preserving its quantum symmetry in Bichon's sense,  it follows that $q^i_k$ and $q^j_l$ commute with each other whenever both $E^i_j$ and $E^k_l$ are nonempty. 
\end{rem}

\begin{prop}\label{banica_iden2}
	Let $\beta$ be a co-action of a CQG $(\mathcal{A},\Delta)$ on $(V,E)$ by preserving its quantum symmetry in Banica's sense.
	The following identity holds:
	$$QW=WQ.$$
	where $Q=(q^i_j)_{i,j\in V}$ is the co-representation matrix of the induced co-action on $C(V)$ and $W$ is the adjacency matrix of $(V,E)$.
\end{prop}

\begin{proof}
	Let $(u^{\sigma}_{\tau})_{\sigma,\tau\in E}$ be the co-representation matrix of $\beta$. \par 
	Let us fix $i,j\in V$. If $i\notin V^s$ or $j\notin V^t$, using equation \ref{source_targ_coeff_presv} it follows that,
	\begin{equation*}
		(QW)^i_j=0=(WQ)^i_j.
	\end{equation*} 
	Hence let us assume $i\in V^s$ and $j\in V^t$. For each $k\in V$ with $W^k_j\neq 0$ we fix an element $\tau_k$ in $E^k_j$. In a similar way, for each $k\in V$ with $W^i_k\neq 0$ we fix an element $\sigma_k$ in $E^i_k$. We observe that,
	\begin{align*}
		(QW)^i_j&=\sum_{k\in V}q^i_k W^k_j=\sum_{\substack{k\in V\\W^k_j\neq 0}}W^k_j(\sum_{\sigma\in E^i}u^{\sigma}_{\tau_{k}})
		=\sum_{\substack{k\in V\\W^k_j\neq 0}}(\sum_{\substack{\sigma\in E^i\\\tau\in E^k_j}}u^{\sigma}_{\tau})=\sum_{\substack{\sigma\in E^i\\ \tau\in E_j}}u^{\sigma}_{\tau};\\
	    (WQ)^i_j&=\sum_{k\in V}W^i_k q^k_j=\sum_{\substack{k\in V\\W^i_k\neq 0}}W^i_k(\sum_{\tau\in E_j}u^{\sigma_k}_{\tau})=\sum_{\substack{k\in V\\W^i_k\neq 0}}(\sum_{\substack{\sigma\in E^i_k\\ \tau\in E_j}}u^{\sigma}_{\tau})=\sum_{\substack{\sigma\in E^i\\ \tau\in E_j}}u^{\sigma}_{\tau}.
	\end{align*}
	Hence our claim is proved.
\end{proof}	
\subsection{The categories $\mathcal{C}^{Ban}_{(V,E)}$, and $\mathcal{C}^{Bic}_{(V,E)}$:} \label{exist_univ_obj}
\begin{defn}
	Let $\beta$ and $\beta'$ be co-actions of two compact quantum groups $(\mathcal{A},\Delta)$ and $(\mathcal{A}',\Delta')$ on $(V,E)$ which preserve its quantum symmetry in Banica's sense. Then $\Phi:(\mathcal{A},\Delta)\rightarrow(\mathcal{A}',\Delta')$, a quantum group homomorphism, is said to intertwine $\beta$ and $\beta'$ if the following diagram commutes:
	\begin{center}
		\begin{tikzcd}
			L^2(E) \arrow[r,"\beta'"]\arrow[d,"\beta"]  & L^2(E)\otimes \mathcal{A}'\\
			L^2(E)\otimes \mathcal{A} \arrow[ru, "id\otimes\Phi"]
		\end{tikzcd}
	\end{center}  
	Let us consider the category $\mathcal{C}^{Ban}_{(V,E)}$ whose objects are triplets $(\mathcal{A}, \Delta_{\mathcal{A}},\beta_{\mathcal{A}})$ where $\beta_{\mathcal{A}}$ is a co-action of a CQG $(\mathcal{A}, \Delta_{\mathcal{A}})$ on $(V,E)$ preserving its quantum symmetry in Banica's sense. Morphisms in this category are quantum group homomorphisms intertwining two such co-actions.
	\par 
	Similarly, we consider the category   $\mathcal{C}^{Bic}_{(V,E)}$ whose objects are compact quantum groups co-acting on $(V,E)$ preserving its quantum symmetry in Bichon's sense and morphisms are quantum group homomorphisms intertwining two similar type co-actions. 
\end{defn}
 Using standard techniques from the theory of compact quantum groups one can show that both these categories admit universal objects (for details, see \cite{asfaq2023thesis}) namely $Q^{Ban}_{(V,E)}$ and $Q^{Bic}_{(V,E)}$. Algebraic descriptions of these two CQGs are given below:   
\begin{defn}\label{Q^Ban_def}
The \textbf{universal compact quantum group associated with a multigraph $(V,E)$}, $Q^{Ban}_{(V,E)}$  is the universal C* algebra generated by the elements of the matrix $(u^{\sigma}_{\tau})_{\sigma,\tau\in E}$ satisfying the following relations:
	\begin{enumerate}
		\item The matrices  $U:=(u^{\sigma}_{\tau})_{\sigma,\tau\in E}$ and $\overline{U}:=(u^{\sigma*}_{\tau})_{\sigma,\tau\in E}$ are both unitary, that is, 
		\begin{align*}
		\sum_{\tau\in E}u^{\sigma_1}_{\tau}u^{\sigma_2*}_{\tau}=\delta_{\sigma_1,\sigma_2}1\quad&\text{and}\quad \sum_{\tau\in E}u^{\tau*}_{\sigma_1}u^{\tau}_{\sigma_2}=\delta_{\sigma_1,\sigma_2}1;\\
		\quad \sum_{\tau\in E}u^{\sigma_1*}_{\tau}u^{\sigma_2}_{\tau}=\delta_{\sigma_1,\sigma_2}1\quad&\text{and}\quad \sum_{\tau\in E}u^{\tau}_{\sigma_1}u^{\tau*}_{\sigma_2}=\delta_{\sigma_1,\sigma_2}1
		\end{align*}
		for all $\sigma_1,\sigma_2\in E$.
		\item $\sum_{\tau\in E}u^{\sigma}_{\tau}=1$ for all $\sigma\in E$.
		\item Let $k\in V^s$. Then for all $\sigma_1,\sigma_2\in E$,
		\begin{align*}
		\sum_{\tau\in E^k}u^{\sigma_1}_{\tau}u^{\sigma_2*}_{\tau}&=0\quad\text{if}\quad\sigma_1\neq\sigma_2;\\
		\sum_{\tau\in E^k}u^{\sigma_1}_{\tau}u^{\sigma_1*}_{\tau}&=\sum_{\tau\in E^k}u^{\sigma_2}_{\tau}u^{\sigma_2*}_{\tau}\quad\text{if}\quad s(\sigma_1)=s(\sigma_2).
		\end{align*}
		\item Let $l\in V^t$. Then for all $\sigma_1,\sigma_2\in E$,
		\begin{align*}
		\sum_{\tau\in E_l}u^{\sigma_1*}_{\tau}u^{\sigma_2}_{\tau}&=0\quad\text{if}\quad\sigma_1\neq\sigma_2;\\
		\sum_{\tau\in E_l}u^{\sigma_1*}_{\tau}u^{\sigma_1}_{\tau}&=\sum_{\tau\in E_l}u^{\sigma_2*}_{\tau}u^{\sigma_2}_{\tau}\quad\text{if}\quad t(\sigma_1)=t(\sigma_2).
		\end{align*}
		\item Let $i\in V^s\setminus V^t, j\in V^t\setminus V^s$ and $k\in V^s\cap V^t$. Then for all $\sigma_1\in E^i, \sigma_2\in E_j, \tau_1\in E^k$ and $\tau_2\in E_k$,
		\begin{equation*}
		u^{\sigma_1}_{\tau_1}=0\quad\text{and}\quad u^{\sigma_2}_{\tau_2}=0.
		\end{equation*} 
		\item Let $i,k\in V^s\cap V^t$. Then for all $\sigma_1\in E^i$ and $\sigma_2\in E_i$,
		\begin{align*}
		\sum_{\tau\in E^k}u^{\sigma_1}_{\tau}u^{\sigma_1*}_{\tau}=\sum_{\tau\in E_k}u^{\sigma_2*}_{\tau}u^{\sigma_2}_{\tau}.
		\end{align*}		
	\end{enumerate} 
The \textbf{quantum automorphism group of $(V,E)$ in Bichon's sense} $Q^{Bic}_{(V,E)}$ is given by,
\begin{equation*}
	Q^{Bic}_{(V,E)}=\bigslant{Q^{Ban}_{(V,E)}}{<u^\sigma_\tau-u^{\sigma*}_{\tau}, u^\sigma_\tau-u^{\sigma2}_{\tau} >}
\end{equation*}
where $<u^\sigma_\tau-u^{\sigma*}_{\tau},  u^\sigma_\tau-u^{\sigma2}_{\tau}>$ is a two sided C* ideal in $Q^{Ban}_{(V,E)}$ generated by the set of elements $\{u^\sigma_{\tau}-u^{\sigma*}_{\tau}, u^\sigma_\tau-u^{\sigma2}_{\tau}|\sigma,\tau\in E\}$.
\end{defn}
\begin{rem}
Abelianisation of the CQG $Q^{Bic}_{(V,E)}$ essentially gives us $C(G^{aut}_{(V,E)})$ where $G^{aut}_{(V,E)}$ is the group of classical automorphisms of $(V,E)$ (see definition \ref{class_aut_mult_def}). However, abelianisation of $Q^{Ban}_{(V,E)}$ produces a group bigger than $G^{aut}_{(V,E)}$, that is, the algebra $C(G^{aut}_{(V,E)})$ is a proper quotient of abelianised $Q^{Ban}_{(V,E)}$ for any genuine multigraph $(V,E)$ (see example 1 in section \ref{chap_examp_appl}).
\end{rem}

\subsection{Restricted orthogonality:} In order to capture only automorphisms in quantum sense and provide a true generalisation of Banica's notion of quantum symmetry in the context of multigraphs, we consider a subcategory $\mathcal{C}^{Ban}_{(V,E)}$ by imposing a further condition namely \textbf{restricted orthogonality}. 
\begin{defn}\label{maindef_symmetric}
Let $\beta$ be a co-action of a CQG $(\mathcal{A},\Delta)$ on $(V,E)$ \textbf{preserving its quantum symmetry in Banica's sense}. Then $\beta$ is said to preserve \textbf{quantum symmetry of $(V,E)$ in our sense} if the following holds:
\begin{align*}
	(Ad_{\beta})^{ij}_{kl}(D_{ij})&\subseteq D_{kl}\otimes \mathcal{A}\\
	\text{and}\quad	(Ad_{\overline{\beta}})^{ij}_{kl}(D_{ij})&\subseteq D_{kl}\otimes \mathcal{A}.
\end{align*}
 for all $i,j,k,l\in V$ with $E^i_j\neq\phi$ and $E^k_l\neq\phi$. The maps $(Ad_{\beta})^{ij}_{kl}:M_{ij}\rightarrow M_{kl}\otimes\mathcal{A}$ (similarly also $(Ad_{\overline{\beta}})^{ij}_{kl}$) are defined by,
\begin{align*}
(Ad_{\beta})^{ij}_{kl}(T)&=(p_{kl}\otimes 1)Ad_{\beta}(T)(p_{kl}\otimes 1),\quad T\in M_{ij}.
\end{align*}
\end{defn}
We have the following algebraic characterisation of \textbf{restricted orthogonality}.
\begin{prop}\label{ortho_lem}
	Let $\beta$ be a co-action of a CQG $(\mathcal{A},\Delta)$ on $(V,E)$ which preserves its quantum symmetry in Banica's sense. Let  $i,j,k,l\in V$ be such that $E^i_j\neq\phi$ and $E^k_l\neq\phi$. Then the following are equivalent:
	\begin{enumerate}
		\item  $(Ad_{\beta})^{kl}_{ij}(D_{kl})\subseteq D_{ij}\otimes \mathcal{A}$ and $(Ad_{\overline{\beta}})^{kl}_{ij}(D_{kl})\subseteq D_{ij}\otimes \mathcal{A}$.
		\item For all $\sigma_1\neq\sigma_2\in E^i_j$ and $\tau\in E^k_l$,
		$$u^{\sigma_1}_{\tau}{u^{\sigma_2}_{\tau}}^*=0\quad\text{and}\quad {u^{\sigma_1}_{\tau}}^*u^{\sigma_2}_{\tau}=0 $$
		where $(u^{\sigma}_{\tau})_{\sigma,\tau\in E}$ is the co-representation matrix of $\beta$. 
	\end{enumerate}
	
\end{prop}
\begin{proof}
	We observe that,  for any  $T\in M_{ij}$, $T\in D_{ij}$ if and only if $\{\chi_{\sigma}|\sigma\in E^i_j\}$ is a set of eigenvectors for $T$.\par 
	
	Let us fix $\tau\in E^k_l$ and $\sigma_2\in E^i_j$. We observe the following identities:
	\begin{align*}
		(Ad_{\beta})^{kl}_{ij}(p_{\tau})(\chi_{\sigma_2}\otimes 1)&=\sum_{\sigma_1\in E^i_j}\chi_{\sigma_1}\otimes u^{\sigma_1}_{\tau}{u^{\sigma_2}_{\tau}}^*;\\
		(Ad_{\overline{\beta}})^{kl}_{ij}(p_\tau)(\chi_{\sigma_2}\otimes 1)&=\sum_{\sigma_1\in E^i_j}\chi_{\sigma_1}\otimes {u^{\sigma_1}_{\tau}}^*u^{\sigma_2}_{\tau}.
	\end{align*}
	From the observation mentioned in the beginning of the proof, the equivalence follows. 
\end{proof}
In the next proposition we see that our condition of ``restricted orthogonality" can not be relaxed any further.
\begin{prop}
Let $\beta$ be a co-action of a CQG $(\mathcal{A},\Delta)$ on $(V,E)$ preserving quantum symmetry of $(V,E)$ in Banica's sense. If $\beta$ satisfies ``complete orthogonality" that is, either
\begin{equation*}
   Ad_{\beta}(D)\subseteq D\otimes\mathcal{A} \quad\text{or}\quad Ad_{\overline{\beta}}(D)\subseteq D\otimes  \mathcal{A},
\end{equation*}
then $\beta$ is a quantum permutation on the edge set $E$, that is, $\beta$ preserves quantum symmetry of $(V,E)$ in Bichon's sense.
\end{prop}
\begin{proof}
We observe that, For $T\in B(L^2(E))$, $T\in D$ if and only if $\{\chi_{\tau}|\tau\in E\}$ is the complete set of eigenvectors of $T$.\\
Let $(u^{\sigma}_{\tau})_{\sigma,\tau\in E}$ be the co-representation matrix of $\beta$.
For $\tau,\sigma_2\in E$, we note that,
\begin{align*}
Ad_{\beta}(p_\tau)(\chi_{\sigma_2}\otimes 1)=\beta(p_{\tau}\otimes 1)(\sum_{\tau'\in E}\chi_{\tau'}\otimes u^{\sigma_2*}_{\tau'})
=\sum_{\sigma_1\in E}\chi_{\sigma_1}\otimes u^{\sigma_1}_{\tau}{u^{\sigma_2*}_{\tau}}.
\end{align*}
Using observation mentioned in the beginning we conclude that,
\begin{align}\label{complete_orthogonality_charac}
Ad_{\beta}(D)\subseteq D\otimes \mathcal{A}\iff u^{\sigma_1}_{\tau}u^{\sigma_2*}_{\tau}=0 \quad\text{when}\quad\sigma_1\neq\sigma_2.
\end{align}
Using observation \ref{complete_orthogonality_charac} and the fact that each row and each column of the matrix $(u^{\sigma}_{\tau})_{\sigma,\tau\in E}$ adds up to $1$, it follows that,
\begin{align*}
u^{\sigma_1}_{\tau}=u^{\sigma_1}_{\tau}(\sum_{\sigma_2\in E}u^{\sigma_2*}_{\tau})=u^{\sigma_1}_{\tau}u^{\sigma_1*}_{\tau}\quad\text{where}\quad \sigma_1,\tau\in E.
\end{align*}
Using spectral calculus for normal operators, it follows that $u^{\sigma_1}_{\tau}$ is a projection making $(u^{\sigma}_{\tau})_{\sigma,\tau\in E}$ a quantum permutation matrix. Therefore $\beta$ preserves quantum symmetry of $(V,E)$ in Bichon's sense. The case when we consider $\overline{\beta}$ instead of $\beta$ can be dealt using similar arguments.
\end{proof}
\subsection*{\textbf{The category $\mathcal{C}^{sym}_{(V,E)}$:}}
The category $\mathcal{C}^{sym}_{(V,E)}$ is a full subcategory of $\mathcal{C}^{Ban}_{(V,E)}$ whose objects are CQGs co-acting on $(V,E)$ preserving its quantum symmetry in our sense, that is, they satisfy ``restricted orthogonality" and morphisms are quantum group homomorphisms intertwining similar type co-actions. As this new restriction does not behave well with the co-product of the ambient quantum group $Q^{Ban}_{(V,E)}$, it is still not clear whether for an arbitrary multigraph, the category $\mathcal{C}^{sym}_{(V,E)}$ admits a universal object. However, under certain restrictions on the multigraph $(V,E)$, the mentioned category does admit a universal object which is a corollary to the theorem \ref{beta_alpha2_co-action_corres} in the next section. Moreover, theorem \ref{beta_alpha2_co-action_corres} also asserts that for any multigraph $(V,E)$, the universal commutative CQG in the category $\mathcal{C}^{sym}_{(V,E)}$ is nothing but $C(G^{aut}_{(V,E)})$ where $G^{aut}_{(V,E)}$ is the group of classical automorphisms of $(V,E)$.
\subsection{\textbf{Consequences of restricted orthogonality:}}
\begin{thm}\label{beta_alpha2_co-action_corres}
Let $\beta$ be a co-action of the CQG $(\mathcal{A},\Delta)$ on $(V,E)$ preserving its quantum symmetry in our sense and $\alpha$ be the induced permutation on the set of vertices $V$. The map $\beta$ is quantum permutation of the the edge set $E$ if and only if $\alpha^{(2)}$ is a quantum permutation of the set $\overline{E}$ where $(V,\overline{E},w)$ is the underlying weighted single edged graph of $(V,E)$ (see definition \ref{under_weighted_simp}). 
\end{thm}
\begin{proof}
Throughout the proof, $(u^{\sigma}_{\tau})_{\sigma,\tau\in E}$ will be the co-representation matrix of $\beta$ and $(q^i_j)_{i,j\in V}$ will be the co-representation matrix of the induced co-action $\alpha$ on $C(V)$. From theorem \ref{wt_preserve} and proposition \ref{banica_iden2}, it further follows that $$\alpha^{(2)}(L^2(\overline{E}))\subseteq L^2(\overline{E})\otimes\mathcal{A}.$$ 

The co-representation matrix of the restricted action $\alpha^{(2)}|_{L^2(\overline{E})}$ is given by $$(q^i_kq^j_l)_{(i,j),(k,l)\in \overline{E}}.$$
If $\alpha$ is a quantum permutation on the edge set $E$, using (3) of proposition \ref{banica_iden1}, it follows that the matrix $(q^i_kq^j_l)_{(i,j),(k,l)\in \overline{E}}$ is a quantum permutation matrix making $\alpha^{(2)}$ a quantum permutation of $\overline{E}$.\par
Conversely, let us assume $\alpha^{(2)}$ is a co-action on the algebra $C(\overline{E})$. For $\tau\in E^i_j$ and $\tau'\in E^k_l$, using proposition \ref{banica_iden1} and proposition \ref{ortho_lem} we observe the following relations:
	\begin{align*}
		\sum_{\sigma\in E^i_j}{u^{\sigma}_{\tau'}}^*u^{\sigma}_{\tau'}&=(\sum_{\sigma\in E^i_j}u^{\sigma}_{\tau'})^*(\sum_{\sigma\in E^i_j}u^{\sigma}_{\tau'})\notag
		=q^j_lq^i_kq^j_l=q^j_lq^i_k;\\
		\sum_{\sigma\in E^i_j}u^{\sigma}_{\tau'}{u^{\sigma}_{\tau'}}^*&=(\sum_{\sigma\in E^i_j}u^{\sigma}_{\tau'})(\sum_{\sigma\in E^i_j}u^{\sigma}_{\tau'})^*\notag
		=q^i_kq^j_lq^i_k=q^j_lq^i_k.
	\end{align*}
	Using above identities it follows that,
	\begin{align*}
		u^{\tau}_{\tau'}{u^{\tau}_{\tau'}}^*u^{\tau}_{\tau'}&=u^{\tau}_{\tau'}(\sum_{\sigma\in E^i_j}{u^{\sigma}_{\tau'}}^*u^{\sigma}_{\tau'})
		=u^{\tau}_{\tau'}(q^j_lq^i_k)
		=u^{\tau}_{\tau'}(\sum_{\sigma\in E^i_j}{u^{\sigma}_{\tau'}}^*)=u^{\tau}_{\tau'}{u^{\tau}_{\tau'}}^*,\\
		\quad u^{\tau}_{\tau'}{u^{\tau}_{\tau'}}^*u^{\tau}_{\tau'}&=(\sum_{\sigma\in E^i_j}u^{\sigma}_{\tau'}u^{\sigma*}_{\tau'})u^{\tau}_{\tau'}=(q^j_l q^i_k)u^{\tau}_{\tau'}=(\sum_{\sigma\in E^i_j}u^{\sigma*}_{\tau'})u^{\tau}_{\tau'}=u^{\tau*}_{\tau'}u^{\tau}_{\tau'}.
	\end{align*}
	Therefore we have, 
	$$u^{\tau}_{\tau'}{u^{\tau}_{\tau'}}^*={u^{\tau}_{\tau'}}^*u^{\tau}_{\tau'}=u^{\tau}_{\tau'}{u^{\tau}_{\tau'}}^*u^{\tau}_{\tau'}.$$
	
	Using spectral calculus for normal operators, we conclude that $u^{\tau}_{\tau'}$ is a projection. 
    As $\tau$ and $\tau'$ were arbitrary, coefficients of the matrix $(u^{\tau}_{\tau'})_{\tau,\tau'\in E}$ are projections. From (4) of definition \ref{maindef} it further follows that coefficients of each row and each column add up to $1$ making $(u^{\tau}_{\tau'})_{\tau,\tau'\in E}$ a quantum permutation matrix and $\alpha$ a quantum permutation of the edge set $E$.  
\end{proof}
We have the following corollary of the above theorem which asserts universal object in $\mathcal{C}^{sym}_{(V,E)}$ for a certain class of multigraphs.
\begin{cor}\label{wr_pdt_proj} For a multigraph $(V,E)$, the two categories $\mathcal{C}^{sym}_{(V,E)}$ and $\mathcal{C}^{Bic}_{(V,E)}$ coincide if and only if the categories $\mathcal{D}^{Ban}_{(V,\overline{E},w)}$ and $\mathcal{D}^{Bic}_{(V,\overline{E},w)}$ coincide where $(V,\overline{E},w)$ is the underlying weighted single edged graph of $(V,E)$ (see definition \ref{under_weighted_simp}). For this class of multigraphs, the universal object in $\mathcal{C}^{Bic}_{(V,E)}$, $Q^{Bic}_{(V,E)}$ is also universal in $\mathcal{C}^{sym}_{(V,E)}$.
\end{cor}
\begin{proof} 
 Using proposition \ref{banica_iden2}, we observe that, for any co-action on a multigraph $(V,E)$, the induced permutation on the vertex set always preserves the weighted symmetry of the underlying weighted graph $(V,\overline{E},w)$. In other words, the induced permutation on the vertex set is a member of $\mathcal{D}^{Ban}_{(V,\overline{E},w)}$. The proof is straightforward using this observation and theorem \ref{beta_alpha2_co-action_corres}.  
  \end{proof}

 Before proceeding further, we describe the notion of uniform components of a multigraph. For a positive integer $m$, a \textbf{uniform multigraph of degree $m$}  is a multigraph where $|E^i_j|=0$ or $m$ for all $i,j\in V$. For a ``non-uniform" multigraph $(V,E)$ and an integer $m$, a \textbf{unifrom component of degree $m$} is a multi-subgraph $(V_m,E_m)$ of $(V,E)$ where  $E_m\subseteq E$ and $V_m\subseteq V$ are given by,
\begin{align*}
E_m&=\{\tau\in E\:|\:\text{cardinality of the set}\:E^{s(\tau)}_{t(\tau)}=m\}\\
V_m&=\{v\in V|v=s(\tau)\:\:\text{or}\:\:v=t(\tau)\:\:\text{for some}\:\:\tau\in E_m\}.
\end{align*}
	It is evident that, $E=\sqcup_{m}E_m$ and  $V=\cup_{m}V_m$. We will write $$(V,E)=\cup_{m}(V_m,E_m).$$ By $V^s_m$ and $V^t_m$, we will mean the sets of initial and final vertices of the multi-subgraph $(V_m,E_m)$.  A multigraph having only one uniform component of degree $m$, is said to be a \textbf{uniform multigraph of degree $m$}.\par 
In the next proposition, we will see that our co-actions preserve uniform components of a multigraph.
\begin{prop}\label{direct_sum}
	Let $\beta:L^2(E)\rightarrow L^2(E)\otimes \mathcal{A}$ be a co-action of a CQG $(\mathcal{A},\Delta)$ on  $(V,E)$ preserving its quantum symmetry in our sense. For any $m\in \mathbb{N}$ with $E_m\neq\phi$, it follows that, $$\beta(L^2(E_m))\subseteq L^2(E_m)\otimes \mathcal{A}.$$
	Conversely, let $\{\beta_m|m\in\mathbb{N}, E_m\neq\phi\}$ be a family of co-actions preserving quantum symmetries (in our sense) of the uniform components $(V_m,E_m)$'s such that induced permutations on the set of vertices agree on the common regions, that is, for all $m\neq n$, $$\alpha_m|_{C(V_m\cap V_n)}=\alpha_{n}|_{C(V_m\cap V_n)}.$$ Then $\beta=\bigoplus_m \beta_m$ is a co-action on $(V,E)$ preserving its quantum symmetry (in our sense). 
\end{prop}
\begin{proof}
	Let $(u^{\sigma}_{\tau})_{\sigma,\tau\in E}$ be the co-representation matrix of $\beta$.
	Let $i,j,k,l\in V$ be such that $E^i_j$ and $E^k_l$ are nonempty and $|E^i_j|\neq|E^k_l|$. It is enough to show that
	\begin{equation}
	u^{\sigma}_{\tau}=0\quad\text{for all}\quad\sigma\in E^i_j,\tau\in E^k_l.
	\end{equation} 
	Using propositions \ref{banica_iden1}, \ref{banica_iden2} and \ref{ortho_lem} we observe that,
	\begin{align*}
	\|u^{\sigma}_{\tau}\|^2=\|{u^{\sigma}_{\tau}}^*u^{\sigma}_{\tau}\|
	=\|{u^{\sigma}_{\tau}}^*(\sum_{\sigma_1\in E^i_j}u^{\sigma_1}_{\tau})\|
	=\|{u^{\sigma}_{\tau}}^*q^i_kq^j_l \|=0.                
	\end{align*}
	Conversely, the proof is done by using corollary \ref{bimod_prop_equiv_cor} repeatedly. 
	As $\alpha_m|_{C(V_m\cap V_n)}=\alpha_n|_{C(V_m\cap V_n)}$ for all $m\neq n\in \mathbb{N}$, there exists a co-action $\alpha:C(V)\rightarrow C(V)\otimes\mathcal{A}$ such that,
	\begin{equation*}
		\alpha (\chi_i)=\alpha_m(\chi_i)\quad\text{where}\quad i\in V_m.
	\end{equation*}
	Let $(q^i_k)_{i,k\in V}$ be the co-representation matrix of $\alpha$. As we clearly have  $\alpha(C(V_m\cap V_n))\subseteq C(V_m\cap V_n)\otimes \mathcal{A}$ for all $m\neq n$, it follows that, 
	\begin{align}\label{V_m_preserve}
		q^k_i&=0\quad\text{where}\quad  i\notin\ V_m\cap V_n\:\:\text{and}\:\: k\in V_m\cap V_n.
	\end{align}
 Let $\tau\in E_m$ and $i\in V_n$ for some nonzero integers $m$ and $n$ such that $m\neq n$. We observe that if $i\in V_m\cap V_n$ then 
	\begin{equation*}
		\alpha(\chi_i)\beta(\chi_\tau)=\alpha_m(\chi_i)\beta_m(\chi_\tau)=\beta_m(\chi_i.\chi_\tau)=\beta(\chi_i.\chi_\tau)
	\end{equation*}
	and if $i\in V_n\setminus V_m$, then using equation \ref{V_m_preserve} we observe that,
	\begin{align*}
		\alpha(\chi_i)\beta(\chi_\tau)&=\alpha_n(\chi_i)\beta_m(\chi_\tau)\\
		&=(\sum_{\substack{k\in V_n\\k\notin V_m\cap V_n}}\chi_{k}\otimes q^k_i)(\sum_{\sigma\in E_m}\chi_{\sigma}\otimes u^{\sigma}_{\tau})\\
		&=0\quad (\text{as $\chi_k.\chi_\sigma=0$ for all $k\in V_n\setminus V_m$})\\
		&=\beta(\chi_{i}.\chi_{\tau}).
	\end{align*}
	Using similar arguments, it also follows that,
	\begin{equation*}
		\beta(\chi_{\tau})\alpha(\chi_i)=\beta(\chi_{\tau}.\chi_i)\quad\text{for all}\quad i\in V,\tau\in E. 
	\end{equation*}
 We further observe that, 
 \begin{equation*}
    \beta(\sum_{\tau\in E}\chi_{\tau})=\sum_{\substack{m\in\mathbb{N}\\E_m\neq\phi}}\beta_{m}(\sum_{\tau\in E_m}\chi_{\tau})=\sum_{\substack{m\in\mathbb{N}\\E_m\neq\phi}}\sum_{\tau\in E_m}\chi_{\tau}\otimes 1=\sum_{\tau\in E}\chi_{\tau}\otimes 1.
    \end{equation*}
 Using corollary \ref{bimod_prop_equiv_cor} it follows that $\beta$ is co-action on $(V,E)$ preserving its quantum symmetry in Banica's sense. To show that $\beta$ satisfies ``restricted orthogonality", we first observe that co-representation matrix of $\beta$ is direct sum of co-representation matrices of $\beta_m$'s. As algebraic relations related to ``restricted orthogonality" (see proposition \ref{ortho_lem}) are satisfied by matrix coefficients of $\beta_m$'s it is easy to see that, same holds for the matrix coefficients of $\beta$ making it a quantum symmetry preserving co-action on $(V,E)$ in our sense.
\end{proof}
In light of above discussion, it is worthwhile to look more into co-actions on uniform multigraphs.

\subsubsection*{\textbf{Co-actions on uniform multigraphs:}}
We start with the description of \textbf{edge-labeling} of a multigraph.

\begin{defn}\label{uni_mult}
	Let $(V,E)$ be a multigraph. For each $k,l\in V$ such that $E^k_l\neq\phi$, let us consider a bijection $\mu_{kl}:\{1,..,m\}\rightarrow E^k_l$ where $|E^k_l|=m$. This set of bijections $\{\mu_{kl}|E^k_l\neq\phi\}$ is said to be an \textbf{edge-labeling of the multigraph $(V,E)$}. Once an edge-labeling is fixed, any $\tau\in E$ can be written as
	\begin{equation*}
		\tau=(k,l)r \quad\text{where}\quad s(\tau)=k,\: t(\tau)=l\:\:\text{and}\:\:1\leq r\leq |E^k_l|.
	\end{equation*}
\end{defn}
\begin{rem}
This method of labeling the edges in a multigraph has been described as a \textbf{representation of a multigraph} in \cite{asfaq2023thesis}. Despite the difference in terminology here and in \cite{asfaq2023thesis}, they necessarily mean the same thing. 
\end{rem}
For proceeding further we will be needing the following technical lemma:

\begin{lem}\label{tech_lem2}
	Let $\{A_i|i=1,2,..,n\}$ be a set of positive operators on a Hilbert space $H$ such that $A_iA_j=0$ when $i\neq j$.  We define $T=\sum_{i=1}^nA_i$. For $i\in\{1,2,..,n\}$, let $p_i$ and $P_T$ be range projections of $A_i$ and $T$, that is, orthogonal projections onto the closures of ranges of $A_i$ and $T$ respectively. Then the following identities are true:
	\begin{enumerate}
		\item $p_ip_j=0\quad \text{when}\quad i\neq j$.
		\item $A_i=p_iT=Tp_i\quad\text{for all}\quad i=1,2,..,n.$
		\item $\sum_{i=1}^np_i=P_T$.
	\end{enumerate}
	
\end{lem}
\begin{proof}
	To prove (1), we observe that, for $\xi,\eta\in H$,
	\begin{equation*}
		<A_i(\xi),A_j(\eta)>=<A_jA_i(\xi),\eta>=0\quad\text{for}\quad i\neq j.
	\end{equation*}
	Therefore range of $p_i$ is orthogonal to range of $p_j$ whenever $i\neq j$ and (1) follows.
	\par
	We further observe that,
	\begin{align*}
		p_iA_j=p_ip_jA_j&=0\quad 
		\text{and}\quad A_jp_i={(p_iA_j)}^*=0\quad \text{where}\quad i\neq j\\
	\text{and therefore}\quad	p_iT&=p_i(\sum_{j=1}^nA_j)=p_iA_i=A_i,\\
		Tp_i&=(\sum_{j=1}^nA_j)p_i=A_ip_i={(p_iA_i)}^*=A_i.
	\end{align*}
	Hence (2) is proved.
	\par 
	To prove claim (3), it is enough to observe that
	\begin{equation*}
		\overline{Range(T)}=\oplus_{i=1}^n\overline{Range(A_i)}
	\end{equation*}
	where the direct sum is an orthogonal direct sum.
\end{proof}

\begin{nota}
	For a C* algebra $\mathcal{A}$, let us denote its universal enveloping Von-Neumann algebra by $\overline{\mathcal{A}}$.
	\par 
	For the rest of this subsection,  we consider $(V,E)$ to be a uniform multigraph of degree $m$ with its edges labeled (see definition \ref{uni_mult}). There is a co-action $\beta$ of a CQG $(\mathcal{A},\Delta)$ on $(V,E)$ preserving its quantum symmetry in our sense. The matrices $(u^{(i,j)r}_{(k,l)s})_{(i,j)r,(k,l)s\in E}$ and $(q^i_j)_{i,j\in V}$ will be the co-representation matrices of $\beta$ and  its induced permutation on the vertex set $V$.
\end{nota}
\begin{prop}\label{struct_thm1}
	Let $i,j,k,l$ be in $V$ such that $E^i_j$ and $E^k_l$ are nonempty. Then there exists a projection valued matrix $(p^{(i,j)r}_{(k,l)s})_{r,s=1,..,m}\in M_m(\mathbb{C})\otimes\overline{\mathcal{A}}$ such that the following holds:
	\begin{equation*}
		u^{(i,j)r}_{(k,l)s}{u^{(i,j)r}_{(k,l)s}}^*=p^{(i,j)r}_{(k,l)s}q^i_kq^j_lq^i_k.
	\end{equation*}
	Here $p^{(i,j)r}_{(k,l)s}$'s are the range projections of $u^{(i,j)r}_{(k,l)s}$ satisfying the following ``quantum permutation like relations":
	\begin{enumerate}
		\item For $r, r'$ and $s\in\{1,2,..,m\},\:\:\:p^{(i,j)r}_{(k,l)s}p^{(i,j)r'}_{(k,l)s}=\delta_{r,r'}p^{(i,j)r}_{(k,l)s}$.
		\item For $r,s$ and $s'\in\{1,2,..,m\},\:\:\:p^{(i,j)r}_{(k,l)s}p^{(i,j)r}_{(k,l)s'}=\delta_{s,s'}p^{(i,j)r}_{(k,l)s}$.
		\item $\sum_{s=1}^mp^{(i,j)r}_{(k,l)s}=\sum_{r=1}^{m}p^{(i,j)r}_{(k,l)s}=P_{q^i_kq^j_lq^i_k}$ where $P_{q^i_kq^j_lq^i_k}$ is the range projection of $q^i_kq^j_lq^i_k$. 
	\end{enumerate} 
	
\end{prop}
\begin{proof}
	Using proposition (\ref{banica_iden1}) and proposition (\ref{ortho_lem}) we observe that,
	\begin{align*}
		\sum_{s=1}^mu^{(i,j)r}_{(k,l)s}{u^{(i,j)r}_{(k,l)s}}^*&=(\sum_{s=1}^mu^{(i,j)r}_{(k,l)s})(\sum_{s=1}^mu^{(i,j)r}_{(k,l)s})^*\\
		&=q^i_kq^j_l{(q^i_kq^j_l)}^*=q^i_kq^j_lq^i_k.
	\end{align*}
	As 	$u^{(i,j)r}_{(k,l)s}{u^{(i,j)r}_{(k,l)s}}^*$'s  are positive operators, using (2) of lemma \ref{tech_lem2} we conclude that
	\begin{equation*}
		u^{(i,j)r}_{(k,l)s}{u^{(i,j)r}_{(k,l)s}}^*=p^{(i,j)r}_{(k,l)s}q^i_kq^j_lq^i_k
	\end{equation*}
	where $p^{(i,j)r}_{(k,l)s}$ is range projection of $u^{(i,j)r}_{(k,l)s}{u^{(i,j)r}_{(k,l)s}}^*$ which is same as the range projection of $u^{(i,j)r}_{(k,l)s}$. The quantum permutation like relations among $p^{(i,j)r}_{(k,l)s}$'s follow from the ``orthogonality relations" mentioned in (1) and (3) of lemma \ref{tech_lem2}.
	
\end{proof}
\begin{prop}\label{struct_thm1.1}
	Let $i,j,k,l$ be in $V$ such that $E^i_j$ and $E^k_l$ are nonempty. Then there exists a projection valued valued matrix $(\hat{p}^{(i,j)r}_{(k,l)s})_{r,s=1,..,m}\in M_m(\mathbb{C})\otimes\overline{\mathcal{A}}$ such that the following holds:
	\begin{equation*}
		u^{(i,j)r*}_{(k,l)s}u^{(i,j)r}_{(k,l)s}=\hat{p}^{(i,j)r}_{(k,l)s}q^j_lq^i_kq^j_l.
	\end{equation*}
	Here $\hat{p}^{(i,j)r}_{(k,l)s}$'s are the range projections of $u^{(i,j)r*}_{(k,l)s}$ satisfying the following \textbf{quantum permutation} like relations:
	\begin{enumerate}
		\item For $r, r'$ and $s\in\{1,2,..,m\},\:\:\:\hat{p}^{(i,j)r}_{(k,l)s}\hat{p}^{(i,j)r'}_{(k,l)s}=\delta_{r,r'}\hat{p}^{(i,j)r}_{(k,l)s}$.
		\item For $r,s$ and $s'\in\{1,2,..,m\},\:\:\:\hat{p}^{(i,j)r}_{(k,l)s}\hat{p}^{(i,j)r}_{(k,l)s'}=\delta_{s,s'}\hat{p}^{(i,j)r}_{(k,l)s}$.
		\item $\sum_{s=1}^m\hat{p}^{(i,j)r}_{(k,l)s}=\sum_{r=1}^{m}\hat{p}^{(i,j)r}_{(k,l)s}=P_{q^j_lq^i_kq^j_l}$ where $P_{q^j_lq^i_kq^j_l}$ is the range projection of $q^j_lq^i_kq^j_l$. 
	\end{enumerate} 
	
\end{prop}
\begin{proof}
	Using proposition \ref{banica_iden1} and proposition \ref{ortho_lem} we observe that,
	\begin{align*}
		\sum_{s=1}^mu^{(i,j)r*}_{(k,l)s}{u^{(i,j)r}_{(k,l)s}}&=(\sum_{s=1}^mu^{(i,j)r}_{(k,l)s})^*(\sum_{s=1}^mu^{(i,j)r}_{(k,l)s})\\
		&={(q^i_kq^j_l)}^*q^i_kq^j_l=q^j_lq^i_kq^j_l.
	\end{align*} 
	The claims follow from lemma \ref{tech_lem2} as it did for proposition \ref{struct_thm1}.
\end{proof}
\begin{cor}\label{struct_thm1_cor}
	Let $i,j,k,l,r,s$ be as in proposition \ref{struct_thm1} or proposition \ref{struct_thm1.1}. Then we have the following commutation relations:
	\begin{enumerate}
		\item $p^{(i,j)r}_{(k,l)s}q^i_kq^j_lq^i_k=q^i_kq^j_lq^i_kp^{(i,j)r}_{(k,l)s}$;
		\item $\hat{p}^{(i,j)r}_{(k,l)s}q^j_lq^i_kq^j_l=q^j_lq^i_kq^j_l\hat{p}^{(i,j)r}_{(k,l)s}$;
		\item $p^{(i,j)r}_{(k,l)s}u^{(i,j)r}_{(k,l)s}{u^{(i,j)r}_{(k,l)s}}^*=u^{(i,j)r}_{(k,l)s}{u^{(i,j)r}_{(k,l)s}}^*=u^{(i,j)r}_{(k,l)s}{u^{(i,j)r}_{(k,l)s}}^*p^{(i,j)r}_{(k,l)s}$;
		\item $\hat{p}^{(i,j)r}_{(k,l)s}u^{(i,j)r*}_{(k,l)s}{u^{(i,j)r}_{(k,l)s}}=u^{(i,j)r*}_{(k,l)s}{u^{(i,j)r}_{(k,l)s}}=u^{(i,j)r*}_{(k,l)s}{u^{(i,j)r}_{(k,l)s}}\hat{p}^{(i,j)r}_{(k,l)s}.$
	\end{enumerate}
\end{cor}
\begin{proof}
	(1) and (3) are immediate from  proposition \ref{struct_thm1}, (2) and (4) are immediate from proposition \ref{struct_thm1.1}. 
\end{proof}
\begin{prop}\label{struct_thm2}
	Let $i,j,k,l\in V$ be such that $E^i_j$ and $E^k_l$  are nonempty and $r,s\in\{1,2,..,m\}$. Then we have the following:
	\begin{equation*}
		u^{(i,j)r}_{(k,l)s}=p^{(i,j)r}_{(k,l)s}q^i_kq^j_l=q^i_kq^j_l\hat{p}^{(i,j)r}_{(k,l)s}
	\end{equation*}
	where $p^{(i,j)r}_{(k,l)s}$'s are described in proposition \ref{struct_thm1} and $\hat{p}^{(i,j)r}_{(k,l)s}$'s are described in proposition \ref{struct_thm1.1}.
\end{prop}
\begin{proof}
	From proposition \ref{banica_iden1}, proposition \ref{struct_thm1} and corollary \ref{struct_thm1_cor} we observe that, 
	\begin{align*}
		&(u^{(i,j)r}_{(k,l)s}-p^{(i,j)r}_{(k,l)s}q^i_kq^j_l)(u^{(i,j)r}_{(k,l)s}-p^{(i,j)r}_{(k,l)s}q^i_kq^j_l)^*\\
		=&u^{(i,j)r}_{(k,l)s}{u^{(i,j)r}_{(k,l)s}}^*-u^{(i,j)r}_{(k,l)s}q^j_lq^i_kp^{(i,j)r}_{(k,l)s}-p^{(i,j)r}_{(k,l)s}q^i_kq^j_l{u^{(i,j)r}_{(k,l)s}}^*+p^{(i,j)r}_{(k,l)s}q^i_kq^j_lq^i_kp^{(i,j)r}_{(k,l)s}\\
		=&u^{(i,j)r}_{(k,l)s}{u^{(i,j)r}_{(k,l)s}}^*-u^{(i,j)r}_{(k,l)s}(\sum_{r'=1}^m{u^{(i,j)r'}_{(k,l)s}}^*)p^{(i,j)r}_{(k,l)s}-p^{(i,j)r}_{(k,l)s}(\sum_{r'=1}^mu^{(i,j)r'}_{(k,l)s}){u^{(i,j)r}_{(k,l)s}}^*\\
		&+u^{(i,j)r}_{(k,l)s}{u^{(i,j)r}_{(k,l)s}}^*p^{(i,j)r}_{(k,l)s}\\
		=&u^{(i,j)r}_{(k,l)s}{u^{(i,j)r}_{(k,l)s}}^*-u^{(i,j)r}_{(k,l)s}{u^{(i,j)r}_{(k,l)s}}^*p^{(i,j)r}_{(k,l)s}-p^{(i,j)r}_{(k,l)s}u^{(i,j)r}_{(k,l)s}{u^{(i,j)r}_{(k,l)s}}^*+u^{(i,j)r}_{(k,l)s}{u^{(i,j)r}_{(k,l)s}}^*p^{(i,j)r}_{(k,l)s}\\
		=&0.
	\end{align*}
	We conclude that
	\begin{equation*}
		u^{(i,j)r}_{(k,l)s}-p^{(i,j)r}_{(k,l)s}q^i_kq^j_l=0\quad\text{and hence}\quad	u^{(i,j)r}_{(k,l)s}=p^{(i,j)r}_{(k,l)s}q^i_kq^j_l.
	\end{equation*}
	The second identity follows by using (2) and (4) of corollary \ref{struct_thm1_cor} and similar computation as above.  
\end{proof}
\begin{rem}
	If we consider $\beta$ to be a quantum permutation of the edge set $E$, that is, a co-action on $(V,E)$ preserving its quantum symmetry in Bichon's sense, it follows that $u^{(i,j)r}_{(k,l)s}=p^{(i,j)r}_{(k,l)s}=\hat{p}^{(i,j)r}_{(k,l)s}$.  
\end{rem} 

\subsection{Description of $Q^{Bic}_{(V,E)}$ for uniform multigraphs:}
We  introduce the following notations that we are going to use in this subsection.

\begin{nota}
	Let $(V,E)$ be a uniform multigraph of degree $m$ and $n=|\overline{E}|$ where $(V,\overline{E})$ is the underlying single edged graph of $(V,E)$. We consider $n$ times free product of the quantum permutation group $S^+_m$. We write the canonical inclusion maps of the free product ${S^+_m}^{*n}$ as $\nu_{(i,j)}:S^+_m\rightarrow {S^+_m}^{*n}$ where $(i,j)\in\overline{E}$. Let $(P^r_s)_{r,s=1,..,m}$ be the matrix of generators of $S^+_m$ satisfying quantum permutation relations. We will write,  
	\begin{equation*}
	P^{(i,j)r}_s=\nu_{(i,j)}(P^r_s)\quad\text{where}\quad (i,j)\in \overline{E}\quad\text{and}\quad r,s=1,2,..,m.
	\end{equation*}
\end{nota}

\begin{thm}\label{wr_pdt_thm}
	Let $(V,E)$ be a uniform multigraph of degree $m$. There is a natural co-action of $Q^{Bic}_{(V,\overline{E})}$ on the algebra ${S^+_m}^{*n}$ which is given by  	
	\begin{equation}\label{wr_pdt_co-action}
	\alpha(\nu_{(k,l)}(a))=\sum_{(i,j)\in \overline{E}}\nu_{(i,j)}(a)\otimes x^i_kx^j_l,\quad (k,l)\in \overline{E}, a\in S^+_m.
	\end{equation} 
	where $(x^i_j)_{i,j\in V}$ is the co-representation matrix of the induced co-action of $Q^{Bic}_{(V,\overline{E})}$ on $C(V)$.  
	 It follows that, with respect to the co-action $\alpha$,
	\begin{equation*}
	Q^{Bic}_{(V,E)}\cong S^+_m*_w Q^{Bic}_{(V,\overline{E})}.
	\end{equation*}
\end{thm}
\begin{proof} We fix an edge-labeling of the multigraph $(V,E)$ (see definition \ref{uni_mult}).
	The quantum automorphism group $Q^{Bic}_{(V,\overline{E})}$ is generated by coefficients of the quantum permutation matrix $(x^i_j)_{i,j\in V}$ where $(V,\overline{E})$ is the underlying single edged graph of $(V,E)$. As $x^i_k$ and $x^j_l$ commute with each other for all $(i,j),(k,l)\in\overline{E}$, it follows that $\alpha$ is a co-action of $Q^{Bic}_{(V,\overline{E})}$ on the C* algebra ${S^+_m}^{*n}$. 
	\par 
	  We observe that there is a co-action $\gamma$ of $S^+_m*_w Q^{Bic}_{(V,\overline{E})}$  on the multigraph $(V,E)$ which preserves its quantum symmetry in Bichon's sense. We define $\gamma:C(E)\rightarrow C(E)\otimes (S^+_m*_w Q^{Bic}_{(V,\overline{E})})$ to be,
	\begin{equation}\label{wr_pr_co-action}
	\gamma(\chi_{(k,l)s})=\sum^m_{\substack{r=1\\(i,j)\in\overline{E}}}\chi_{(i,j)r}\otimes P^{(i,j)r}_sx^i_kx^j_l;\quad (k,l)\in \overline{E}\:\:\text{and}\:\:s=1,2,..,m.
	\end{equation}
By universality of $Q^{Bic}_{(V,E)}$, we have a quantum group homomorphism $\Phi:Q^{Bic}_{(V,E)}\rightarrow S^+_m*_w Q^{Bic}_{(V,\overline{E})}$ satisfying
	\begin{equation*}
	\Phi(u^{(i,j)r}_{(k,l)s})=P^{(i,j)r}_sx^i_kx^j_l\quad \text{where}\quad (i,j),(k,l)\in\overline{E}\quad\text{and}\quad r,s=1,2,..,m.
	\end{equation*}
	 Let us denote $(q^i_j)_{i,j\in V}$ to be the co-representation matrix of the induced co-action of $Q^{Bic}_{(V,E)}$ on $C(V)$. Now we  construct the inverse of $\Phi$ to show that it is in fact an isomorphism of compact quantum groups. 
	\par  
	For $(i,j)\in\overline {E}$ and $r,s\in\{1,2,..,m\}$ we define,
	\begin{equation*}
	R^{(i,j)r}_s=\sum_{(k,l)\in\overline{E}}u^{(i,j)r}_{(k,l)s}.
	\end{equation*}
	 We proceed through following claims.\par 
	\textbf{Claim 1:}
	Let $(i,j)\in \overline{E}$. The coefficients of the matrix $(R^{(i,j)r}_s)_{r,s=1,..,m}$ satisfy quantum permutation relations.\par 
	
	We observe that,
	\begin{align*}
	{R^{(i,j)r}_s}^{2}=&R^{(i,j)r}_s={R^{(i,j)r}_s}^*\\ 
	\text{and}\quad\sum_{r=1}^m R^{(i,j)r}_s= \sum_{\substack{r=1\\(k,l)\in\overline{E}}}^m u^{(i,j)r}_{(k,l)s}=&\sum_{(k,l)\in\overline{E}}q^i_kq^j_l=1=\sum_{\substack{s=1\\(k,l)\in\overline{E}}}^mu^{(i,j)r}_{(k,l)s}=\sum_{s=1}^m R^{(i,j)r}_s.
	\end{align*}\par 
	\textbf{Claim 2:}
	For $(i,j),(k,l)\in\overline{E}$ and $r,s\in\{1,2,..,m\}$, we have the following relations:
	\begin{equation*}
	u^{(i,j)r}_{(k,l)s}=R^{(i,j)r}_sq^i_kq^j_l\quad\text{and}\quad R^{(i,j)r}_sq^i_kq^j_l=q^i_kq^j_lR^{(i,j)r}_s.
	\end{equation*}\par  
    We observe that,
	\begin{align*}
	R^{(i,j)r}_sq^i_kq^j_l&=(\sum_{(k',l')\in \overline{E}}u^{(i,j)r}_{(k',l')s})(\sum_{s'=1}^m u^{(i,j)r}_{(k,l)s'})=u^{(i,j)r}_{(k,l)s},\\
	q^i_kq^j_lR^{(i,j)r}_s&=(\sum_{s'=1}^m u^{(i,j)r}_{(k,l)s'})(\sum_{(k',l')\in \overline{E}}u^{(i,j)r}_{(k',l')s})=u^{(i,j)r}_{(k,l)s}.
	\end{align*}
	Hence claim 2 follows.\par

	\textbf{Claim 3:}
	Let $\Delta_{Bic}$ denote the co-product on $Q^{Bic}_{(V,E)}$. The co-product identities in theorem \ref{copdt_wr_pdt} hold, that is,
	\begin{align*}
	\Delta_{Bic}(q^i_j)&=\sum_{k\in V}q^i_k\otimes q^k_j\\
	\text{and}\quad 
	\Delta_{Bic}(R^{(i,j)r}_s)&=\sum_{\substack{s'=1\\(k,l)\in\overline{E}}}^m(R^{(i,j)r}_{s'}\otimes R^{(k,l)s'}_s)(q^i_kq^j_l\otimes 1).
	\end{align*}\par 
	
	The first identity is immediate. To prove the second one we observe that,
	\begin{align*}
	\Delta_{Bic}(R^{(i,j)r}_s)=\Delta_{Bic}(\sum_{(k',l')\in\overline{E}}u^{(i,j)r}_{(k',l')s})
	&=\sum_{(k',l')\in\overline{E}}(\sum_{\substack{s'=1\\(k,l)\in \overline{E}}}^mu^{(i,j)r}_{(k,l)s'}\otimes u^{(k,l)s'}_{(k',l')s})\\
	&=\sum_{\substack{s'=1\\(k,l)\in\overline{E}}}^m R^{(i,j)r}_{s'} q^i_kq^j_l\otimes (\sum_{(k',l')\in\overline{E}}u^{(k,l)s'}_{(k',l')s})\\
	&=\sum_{\substack{s'=1\\(k,l)\in\overline{E}}}^m R^{(i,j)r}_{s'} q^i_kq^j_l\otimes R^{(k,l)s'}_s.
	\end{align*}
	Hence the second identity in claim 3 follows. \par 
	
	Using claim 1, claim 2, claim 3 and universality of free wreath product we get a surjective quantum group homomorphism $\Psi:S^+_m*_w Q^{Bic}_{(V,\overline{E})}\rightarrow Q^{Bic}_{(V,E)}$ such that the following hold:
	\begin{equation*}
	\Psi(x^i_j)=q^i_j\quad\text{and}\quad\Psi(P^{(i',j')r}_s)=R^{(i',j')r}_s
	\end{equation*}
	where $i,j\in V$,$(i',j')\in \overline{E}$ and $r,s=1,2,..,m$.\par 
	
	It is clear that $\Phi$ and $\Psi$ are inverses of each other as it is such on the set of generators. Hence theorem \ref{wr_pdt_thm} is proved. 
\end{proof}

\section{\textbf{Examples and computations}}\label{chap_examp_appl}
In this section we compute quantum automorphism groups of a few selected multigraphs.

\subsubsection*{\textbf{Example 1:}}

\begin{figure}[h]
	\centering
	\rotatebox[origin=c]{45}{
		\begin{tikzpicture}[scale=1.5]
			\node at (0,0) {$\bullet$};
			\node at (0,.2) {\rotatebox{-45}{$a$}};
			\draw[thick, out=135, in=180, looseness=1.2] (0,0) to (0,1) node{\backmidarrow};
			\draw[thick, out=0, in=45, looseness=1.2] (0,1) to (0,0);
			\draw[thick, out=135, in=90, looseness=1.2] (0,0) to (-1,0) node{\downmidarrow};
			\draw[thick, out=-90, in=-135, looseness=1.2] (-1,0) to (0,0);
			\draw[dashed, out=-135, in=180, looseness=1.2] (0,0) to (0,-1) node{\midarrow};
			\draw[dashed, out=0, in=-45, looseness=1.2] (0,-1) to (0,0);
			\draw[dashed, out=45, in=90, looseness=1.2] (0,0) to (1,0) node{\upmidarrow};
			\draw[dashed, out=-45, in=-90, looseness=1.2] (0,0) to (1,0); 
	\end{tikzpicture}}
	\caption{A multigraph with $n$ loops on a single vertex}
	\label{4_loops}
\end{figure}
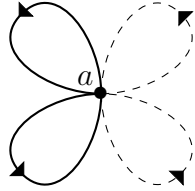
Let us consider the multigraph in figure \ref{4_loops} where the vertex set has a single element $a$ and edge set $E$ has $n$ number of loops, that is, $n$ number of edges with single endpoint vertex $a$. The universal CQG associated with $(V,E)$, $Q^{Ban}_{(V,E)}$ is the universal C* algebra generated by coefficients of the matrix $(u^{\sigma}_{\tau})_{\sigma,\tau\in E}$ satisfying the following relations:
\begin{align*}
	\sum_{\tau\in E}u^{\sigma_1}_{\tau}u^{\sigma_2*}_{\tau}=\delta_{\sigma_1,\sigma_2}1,&\quad\sum_{\tau\in E}u^{\sigma_1*}_{\tau}u^{\sigma_2}_{\tau}=\delta_{\sigma_1,\sigma_2}1,\\
	\sum_{\tau\in E}u^{\tau}_{\sigma_1}u^{\tau*}_{\sigma_2}=\delta_{\sigma_1,\sigma_2}1,&\quad \sum_{\tau\in E}u^{\tau*}_{\sigma_1}u^{\tau}_{\sigma_2}=\delta_{\sigma_1,\sigma_2}1,\\ 
	\text{and}\quad 
	\sum_{\tau\in E}u^{\sigma_1}_{\tau}&=1
\end{align*} 
where $\sigma_1,\sigma_2\in E$. 
\par 

From corollary \ref{wr_pdt_proj}, it follows that the category $\mathcal{C}^{sym}_{(V,E)}$  admits  universal object which is $Q^{Bic}_{(V,E)}$. Moreover, from theorem \ref{wr_pdt_thm} it follows that $Q^{Bic}_{(V,E)}=S^+_n$ where $S^+_n$ is the quantum permutation group on n elements. 

\subsubsection*{\textbf{Example 2:}}
\begin{figure}[h]
	\centering
	\begin{tikzpicture}[scale=2]
		\node at (-1,0) {$\bullet$};
		\node at (1,0) {$\bullet$};
  \draw[dashed, out=70, in=110, looseness=1.6] (-1,0) edge node{\midarrow} (1,0);
  \draw[dashed, out=-70, in=-110, looseness=1.6] (-1,0) edge node{\backmidarrow} (1,0);
		\draw[thick] (-1,0) edge[bend left] node{\midarrow} (1,0);
		\draw[thick] (-1,0) edge[bend right] node{\backmidarrow} (1,0);
		\draw[dashed, out=60, in=120, looseness=1.2]  (-1,0) to node{\midarrow} (1,0);
		\draw[dashed, out=-60, in=-120, looseness=1.2]  (-1,0) to node{\backmidarrow} (1,0);
		\node at (-1.2,0) {$a$};
		\node at (1.2,0) {$b$};
	\end{tikzpicture}
	\caption{A multigraph with two vertices and $2n$ number of edges.}
	\label{2_vertex_multi}
\end{figure}
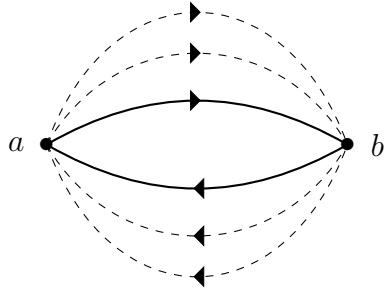

We consider the multigraph $(V,E)$ in figure \ref{2_vertex_multi}. The vertex set has two elements $a$ and $b$, edge set $E$ consists of $n$ edges from $a$ to $b$ and $n$ edges from $b$ to $a.$ Let us fix an edge-labeling for $(V,E)$ (see definition \ref{uni_mult}). Using corollary \ref{wr_pdt_proj} it follows that, $\mathcal{C}^{sym}_{(V,E)}$ admits universal object which is $Q^{Bic}_{(V,E)}$. The quantum automorphism group of the underlying single edged graph $Q^{Ban}_{(V,\overline{E})}$ is $C(\mathbb{Z}_2)$ and is generated by the  coefficients of the following matrix:
$$
\begin{bmatrix}
q&1-q\\
1-q&q

\end{bmatrix}
$$ where $q$ is a projection. Let $(u^{(i,j)r}_{(k,l)s})_{(i,j)r,(k,l)s\in E}$ be the matrix of canonical generators of $Q^{Bic}_{(V,E)}$. For $r,s=1,..,n$ let us define,
\begin{equation*}
P^{(a,b)r}_s=u^{(a,b)r}_{(a,b)s} + u^{(a,b)r}_{(b,a)s}\quad\text{and}\quad P^{(b,a)r}_s=u^{(b,a)r}_{(b,a)s}+ u^{(b,a)r}_{(a,b)s}  
\end{equation*}
We observe that,
\begin{enumerate}
    \item The matrices $(P^{(a,b)r}_s)_{r,s=1,..,n}$ and $(P^{(b,a)r}_s)_{r,s=1,..,n}$ are quantum permutation matrices.
    \item $P^{(a.b)r}_s$ and $P^{(b,a)r}_s$ commute with $q$ for all $r,s=1,..,n$.
\end{enumerate}
Therefore we have, 
\begin{equation*}
    Q^{Bic}_{(V,E)}=C^*\{P^{(a.b)r}_s, P^{(b,a)r}_s, q\: |\: r,s=1,..,n\}= (S^+_n*S^+_n)\otimes C(\mathbb{Z}_2)
\end{equation*}
Moreover. the co-product $\Delta_{Bic}$ on $Q^{Bic}_{(V,E)}$ is given by,
\begin{align*}
    \Delta_{Bic} (q)&=q\otimes q + (1-q) \otimes (1-q)\\
    \Delta_{Bic}(P^{(a,b)r}_s)&=\sum^n_{s'=1} P^{(a,b)r}_sq \otimes P^{(a,b)s'}_s + P^{(a,b)r}_s (1-q) \otimes P^{(b,a)s'}_s\\
    \Delta_{Bic}(P^{(b,a)r}_s)&=\sum^n_{s'=1} P^{(b,a)r}_sq \otimes P^{(b,a)s'}_s + P^{(b,a)r}_s (1-q) \otimes P^{(a,b)s'}_s.
\end{align*}

\subsubsection*{\textbf{Example 3:}}
\begin{figure}[h]
	\centering
	\begin{tikzpicture}[scale=4]
		\node at (-1,0) {$\bullet$};
		\node at (1,0) {$\bullet$};
		\node at (0,1) {$\bullet$};
		\draw[thick] (-1,0) edge node{\backslantmidarrow} (0,1);
		\draw[thick, out=67, in=203, looseness=1.2] (-1,0) edge node{\backslantmidarrow} (0,1);
		\draw[dashed, out=23, in=-113, looseness =1.2] (-1,0) edge node{\backslantmidarrow} (0,1);
		\draw[thick] (-1,0) edge node{\midarrow} (1,0);
		\draw[thick, out=-22, in=-158] (-1,0) edge node{\midarrow} (1,0);
		\draw[dashed, out=22, in=158] (-1,0) edge node{\midarrow} (1,0);
		\draw[thick] (1,0) edge node{\backopslantmidarrow} (0,1);
		\draw[thick, out=113, in=-23, looseness=1.2] (1,0) edge node{\backopslantmidarrow} (0,1);
		\draw[dashed, out=157, in=-67,looseness=1.2] (1,0) edge node{\backopslantmidarrow} (0,1);
		\node at (-1.2,0) {$b$};
		\node at (1.2,0) {$c$};
		\node at (0,1.2) {$a$}; 
	\end{tikzpicture}
	\caption{A multigraph version of a triangle with $n$ edges between two vertices.}
	\label{trian_multi}
\end{figure}
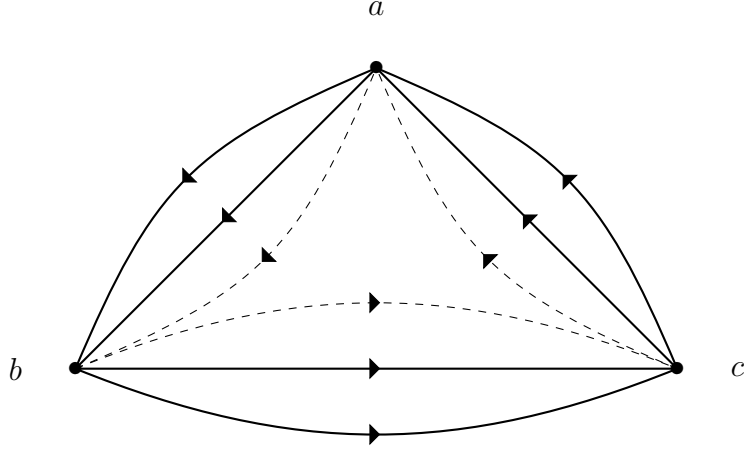
We consider the multigraph $(V,E)$ in figure \ref{trian_multi} where there are three vertices $a,b,c$ and $n$ edges from $a$ to $b$, $b$ to $c$ and $c$ to $a$ respectively. We fix an edge-labeling for $(V,E)$ (see definition \ref{uni_mult}). From corollary \ref{wr_pdt_proj} it follows that $\mathcal{C}^{sym}_{(V,E)}$ admits universal object which is $Q^{Bic}_{(V,E)}$. If $(V,\overline{E})$ is the underlying single edged graph of $(V,E)$, then it follows that $Q^{Ban}_{(V,\overline{E})}$, which is $C(\mathbb{Z}_3)$, is generated by coefficients of the matrix,
$$
\begin{pmatrix}
q_1&q_2&q_3\\
q_3&q_1&q_2\\
q_2&q_3&q_1
\end{pmatrix}
$$
where $q_1, q_2, q_3$ are mutually orthogonal projections and $q_1+q_2+q_3=1$.

Let $(u^{(i,j)r}_{(k,l)s})_{(i,j)r,(k,l)s\in E}$ be the matrix of canonical generators of $Q^{Bic}_{(V,E)}$. For $(i,j)\in \overline{E}$ and $r,s\in \{1,2,..,n\}$, let us define, 
\begin{equation*}
P^{(i,j)r}_s=\sum_{(k,l)\in \overline{E}} u^{(i,j)r}_{(k,l)s}.
 \end{equation*}
The quantum automorphism group $Q^{Bic}_{(V,E)}$ is generated by the following set
of generators:
\begin{equation*}
\cup_{(i,j)\in\overline{E}}\{P^{(i,j)r}_s|r,s=1,2,..,n\}\cup \{q_1,q_2,q_3\}
\end{equation*}
such that the following conditions hold:
\begin{enumerate}
\item $q_1,q_2$ and $q_3$ are mutually orthogonal projections such that $q_1+q_2+q_3=1$.
\item For each $(i,j)\in\overline{E}$, the matrix $(P^{(i,j)r}_s)_{r,s=1,2,..,n}$ is a \textbf{quantum permutation matrix}.
\item  $P^{(i,j)r}_s$ commutes with $q_k$ for all $k=1,2,3; (i,j)\in \overline{E}; r,s=1,2,..,n$.
\end{enumerate}
It is clear that as an algebra $Q^{Bic}_{(V,E)}$ is $(S^+_n * S^+_n * S^+_n)\otimes C(\mathbb{Z}_3)$.
Moreover the co-product $\Delta_{Bic}$ on $Q^{Bic}_{(V,E)}$ is given by
\begin{align*}
&\Delta_{Bic}(q_1)=q_1\otimes q_1+q_2\otimes q_3+q_3\otimes q_2,\quad \Delta_{Bic}(q_2)=q_3\otimes q_3 + q_1\otimes q_2 + q_2\otimes q_1,\\ 
&\qquad \qquad \qquad\qquad\Delta_{Bic}(q_3)=q_2\otimes q_2 + q_1\otimes q_3 + q_3\otimes q_1;\\
&\Delta_{Bic}(P^{(a,b)r}_s)=\sum^n_{s'=1}(P^{(a,b)r}_{s'}q_1\otimes P^{(a,b)s'}_s) + (P^{(a,b)r}_{s'}q_2\otimes P^{(b,c)s'}_s) + (P^{(a,b)r}_{s'}q_3\otimes P^{(c,a)s'}_s),\\
&\Delta_{Bic}(P^{(b,c)r}_s)=\sum^n_{s'=1}(P^{(b,c)r}_{s'}q_1\otimes P^{(b,c)s'}_s) + (P^{(b,c)r}_{s'}q_2\otimes P^{(c,a)s'}_s) + (P^{(b,c)r}_{s'}q_3\otimes P^{(a,b)s'}_s),\\
&\Delta_{Bic}(P^{(c,a)r}_s)=\sum^n_{s'=1}(P^{(c,a)r}_{s'}q_1\otimes P^{(c,a)s'}_s) + (P^{(c,a)r}_{s'}q_2\otimes P^{(a,b)s'}_s) + (P^{(c,a)r}_{s'}q_3\otimes P^{(b,c)s'}_s).
\end{align*}

\section{\textbf{Applications}}\label{graphC*algebra}
\subsection{\textbf{Quantum symmetry of graph C* algebras:}}In the context of quantum symmetry, graph C* algebras are interesting objects to study as they are mostly infinite dimensional although function algebras associated with graphs are not. In this subsection, we will see that our notions of quantum symmetry in multigraphs lift to the level of graph C* algebras. We start by recalling the definition of a graph C* algebra associated with a multigraph $(V,E)$. For more details, see \cite{Raeburn2005}, \cite{Brannan2022}, \cite{Pask2006} and references within.
\begin{defn}\label{graph_C*_algebra_def}
	For a finite multigraph $\Gamma=(V,E)$ the graph C* algebra $C^*(\Gamma)$ is the universal C*algebra generated by a set of partial isometries $\{s_\tau|\tau\in E\}$ and a set of mutually orthogonal projections $\{p_i|i\in V\}$ satisfying the following relations among them:
	\begin{enumerate}
		\item $s^*_\tau s_\tau=p_{t(\tau)}$ for all $\tau\in E$ where $t:E\rightarrow V$ is the target map of  $\Gamma$.
		\item $\sum_{\tau\in E^i}s_\tau s^*_\tau=p_i$ for all $i\in V^s$ where $V^s$ is the set of initial vertices in $\Gamma$.
	\end{enumerate}
\end{defn}
We have the following properties of graph C* algebras (subsection 2.1 of \cite{Pask2006}).
\begin{enumerate}
	\item $\sum_{i\in V}p_i=1$ in $C^*(\Gamma)$.
	\item For any $i\in V^s$, $\{s_{\tau}s^*_{\tau}|\tau\in E^i\}$ is a set of mutually orthogonal projections and $s^*_{\tau_1}s_{\tau_2}=0$ for all $\tau_1\neq \tau_2\in E$.
\end{enumerate}
We will be generalising the main result in \cite{Schmidt2018} in our framework of quantum symmetry in multigraphs using similar arguments.
\begin{thm}
	Let $\Gamma=(V,E)$ be a multigraph and $\beta$ be a co-action of a CQG $(\mathcal{A},\Delta)$ on $(V,E)$ preserving its quantum symmetry in Banica's sense (see definition \ref{maindef}). Then $\beta$ induces a co-action  $\beta':C^*(\Gamma)\rightarrow C^*(\Gamma)\otimes\mathcal{A}$ satisfying,
	\begin{align*}
		\beta'(p_i)&=\sum_{k\in V}p_{k}\otimes q^{k}_{i},\\
		\beta'(s_\tau)&=\sum_{\sigma\in E}s_{\sigma}\otimes u^{\sigma}_{\tau}
	\end{align*}
	where $(u^{\sigma}_{\tau})_{\sigma,\tau\in E}$ and $(q^{k}_i)_{k,i\in V}$ are the co-representation matrices of $\beta$ and its induced co-action $\alpha$ on $C(V)$.
	
\end{thm}
\begin{proof}
	For $\tau\in E,i\in V$, let us define $S_\tau, P_v\in C^*(\Gamma)\otimes \mathcal{A}$ by
	\begin{align*}
		S_\tau=\sum_{\sigma\in E}s_{\sigma}\otimes u^{\sigma}_{\tau}\quad\text{and}\quad
		P_i=\sum_{k\in V}p_{k}\otimes q^{k}_{i}.
	\end{align*}
	For $i, j\in V$, we observe that,
	\begin{equation*}
		P_i P_j=\sum_{k\in V}p_k\otimes q^k_i q^k_j=\delta_{i,j}\sum_{k\in V}p_k\otimes q^k_i=\delta_{i,j}P_{i}.
	\end{equation*}
	Hence $\{P_i|i\in V\}$ is a set of mutually orthogonal projections in $C^*(\Gamma)\otimes \mathcal{A}$. Using properties of $C^*(\Gamma)$, we observe that, for $\tau\in E$,
	\begin{align*}
		  	S^*_{\tau}S_{\tau}=\sum_{\sigma_1,\sigma_2\in E}s^*_{\sigma_1}s_{\sigma_2}\otimes u^{\sigma_1*}_{\tau}u^{\sigma_2}_{\tau}
		=\sum_{\sigma\in E}s^*_{\sigma}s_{\sigma}\otimes u^{\sigma*}_{\tau}u^{\sigma}_{\tau}
		&=\sum_{\sigma\in E}p_{t(\sigma)}\otimes u^{\sigma*}_{\tau}u^{\sigma}_{\tau}\\
		&=\sum_{k\in V^t}p_k\otimes \sum_{\sigma\in E_k}u^{\sigma*}_{\tau}u^{\sigma}_{\tau}\\
		&=\sum_{k\in V}p_k\otimes q^k_{t(\tau)}=P_{t(\tau)}.
  \end{align*}
  For $i\in V^s$, it further follows that,
  \begin{align*}
		 \sum_{\tau\in E^i}S_{\tau}S^*_{\tau}=\sum_{\sigma_1,\sigma_2\in E}s_{\sigma_1}s^*_{\sigma_2}\otimes \sum_{\tau\in E^i}u^{\sigma_1}_{\tau}u^{\sigma_2*}_{\tau}
		&=\sum_{\sigma_1,\sigma_2\in E}s_{\sigma_1}s^*_{\sigma_2}\otimes \delta_{\sigma_1,\sigma_2}q^{s(\sigma_1)}_i\\
		&=\sum_{k\in V^s}(\sum_{\sigma\in E^k}s_{\sigma}s^*_{\sigma})\otimes q^k_i\\
		&=\sum_{k\in V}p_k\otimes q^k_i=P_i\label{src_graphC*algebra_rel}.
	\end{align*}
    By universality of $C^*(\Gamma)$, there exists a C* algebra homomorphism $\beta':C^*(\Gamma)\rightarrow C^*(\Gamma)\otimes\mathcal{A}$ such that,
	\begin{equation*}
		\beta'(s_{\tau})=S_{\tau}\quad \text{and} \quad \beta'(p_i)=P_i
	\end{equation*}
	for all $\tau\in E$ and $i\in V$. It remains to show that $\beta'$ is in fact a co-action of $(\mathcal{A},\Delta)$ on $C^*(\Gamma)$. The co-product identity holds as it is easy to check that on the set of generators of $C^*(\Gamma)$. Let us define
	\begin{equation*}
		\mathcal{S}=\text{linear span}\:\beta'(C^*(\Gamma))(1\otimes \mathcal{A})\subseteq C^*(\Gamma)\otimes \mathcal{A}.
	\end{equation*}
	To conclude that $\beta'$ is a co-action, it is enough to show that $\mathcal{S}$ is norm-dense in $C^*(\Gamma)\otimes\mathcal{A}$. We proceed through  following claims:\par
	\textbf{Claim 1:} $p_i\otimes 1,s_\tau\otimes 1,s^*_\tau\otimes 1\in \mathcal{S}$ for all $i\in V,\tau\in E$.\par
	Let $i\in V,\tau\in E$.  We observe that,
	\begin{align*}
		\sum_{j\in V}\beta'(p_j)(1\otimes q^i_j)&=\sum_{l\in V}p_l\otimes (\sum_{j\in V}q^l_jq^i_j)=p_i\otimes \sum_{j\in V}q^i_j=p_i\otimes 1,\\
		\sum_{\sigma\in E}\beta'(s_\sigma)(1\otimes u^{\tau*}_{\sigma})&=\sum_{\sigma'\in E}s_{\sigma'}\otimes (\sum_{\sigma\in E}u^{\sigma'}_{\sigma}u^{\tau*}_{\sigma})=s_{\tau}\otimes 1,\\
		\sum_{\sigma\in E}\beta'(s^*_\sigma)(1\otimes u^{\tau}_{\sigma})&=\sum_{\sigma'\in E}s^*_{\sigma'}\otimes (\sum_{\sigma\in E}u^{\sigma'*}_{\sigma}u^{\tau}_{\sigma})=s^*_{\tau}\otimes 1.\\
	\end{align*}
	In the above computation we have used the fact that $\beta$ and $\overline{\beta}$ both are unitary co-representations on $L^2(E)$.
	As all the elements mentioned in the left are in $\mathcal{S}$, claim 1 follows.
	\\
	\\
	\textbf{Claim 2:} If $x\otimes 1,y\otimes 1\in \mathcal{S}$, then $xy\otimes 1\in \mathcal{S}$.\par
	Let us assume that,
	\begin{equation*}
		x\otimes 1=\sum_{i=1}^n\beta'(e_i)(1\otimes f_i)\quad \text{and} \quad y\otimes 1=\sum_{j=1}^m\beta'(g_j)(1\otimes h_j)
	\end{equation*}
	where $e_i,g_j\in C^*(\Gamma)$ and $f_i,h_j\in\mathcal{A}$ for all $i,j$. We observe that,
	\begin{align*}
		xy\otimes 1=\sum_i\beta'(e_i)(1\otimes f_i)(y\otimes 1)
		&=\sum_i\beta'(e_i)(y\otimes 1)(1\otimes f_i)\\
		&=\sum_{i,j}\beta'(e_i)\beta'(g_j)(1\otimes g_j)(1\otimes f_i)\\
		&=\sum_{i,j}\beta'(e_ig_j)(1\otimes g_jf_i)\:\:\in\mathcal{S}.
	\end{align*}
	Hence claim 2 follows.\par
	From claim 1 and claim 2 it is clear that,
	\begin{equation*}
		C^*(\Gamma)\otimes 1 \subseteq \text{norm closure of}\:\:\mathcal{S}.
	\end{equation*}
	As for any $T\in \mathcal{S}$ and $x\in\mathcal{A}$, $T(1\otimes x)$ is also in  $\mathcal{S}$, we conclude that,
	\begin{equation*}
		C^*(\Gamma)\otimes \mathcal{A} \subseteq \text{norm closure of}\:\:\mathcal{S}.
	\end{equation*}
	Hence our theorem is proved.
\end{proof}

\subsection{\textbf{Quantum symmetry in undirected multigraphs:}}\label{quan_sym_undir_mult}
An undirected multigraph consists of an edge set $E$, a set of vertices $V$ and a range map $r:E\rightarrow \{\{i,j\}\:|\:i,j\in V\}$ where $\{.,.\}$ is an unordered pair of vertices. In our context we describe undirected multigraph as a  ``doubly directed" multigraph with an inversion map which identifies two oppositely directed edges to produce an undirected edge. 
\begin{defn}
    A ``doubly directed" multigraph $(V,E)$ is a multigraph whose adjacency matrix is symmetric, that is, $|E^i_j|=|E^j_i|$ for all $i,j\in V$. An undirected multigraph $(V,E,j)$ is a ``doubly directed" $(V,E)$ with an inversion map $j:E\rightarrow E$ satisfying the following conditions: 
    \begin{enumerate}
        \item $j^2=id_E$.
        \item For all $\sigma\in E$, $j(\sigma)=\sigma$ if  $s(\sigma)=t(\sigma)$. 
        \item For all $\sigma\in E$ $$\quad s(j(\sigma))=t(\sigma)\quad\text{and}\quad t(j(\sigma))=s(\sigma).$$
    \end{enumerate}
\end{defn}
The inversion map $j$ in an undirected multigraph $(V,E,j)$ induces a linear map $J:L^2(E)\rightarrow L^2(E)$ on vector space level. Any quantum symmetry preserving co-action $\beta$ on an undirected multigraph $(V,E,j)$ can be described as a quantum symmetry preserving co-action on the doubly directed multigraph $(V,E)$ such that the linear map  $J$ intertwines $\beta$ and its contragradient $\overline{\beta}$, that is, 
\begin{equation*}
   \beta\circ J=(J\otimes id)\circ\overline{\beta} 
\end{equation*}
It is also enough to start with only a unitary co-representation instead of a bi-unitary one because any unitary satisfying the intertwinement condition mentioned above is essentially a bi-unitary map. We encourage the reader to look into section 4.6 of \cite{asfaq2023thesis} for more discussions about quantum symmetry of undirected multigraphs. In  theorem 4.6.11 of \cite{asfaq2023thesis} we have shown that a quantum symmetry preserving co-action on a directed multigraph do arise from a co-action on the underlying undirected multigraph preserving the set of ``directed" edges. This is also a classical phenomena which demonstrates consistency between our different constructions in the quantum case.

\bibliographystyle{alphaurl}
\bibliography{quantum_symmetry_in_multigraphs}
\end{document}